\documentclass[twoside,11pt]{article}

\usepackage{amssymb,latexsym,amsmath}
\usepackage{epsfig}
\usepackage{caption}

\usepackage{jmlr2e}
 
 \usepackage{amsmath}
\usepackage{amsfonts}
\usepackage{array}
\usepackage{dsfont}
\usepackage{hyperref}
\usepackage{amssymb}

\usepackage{makeidx}
\usepackage{graphicx}
\graphicspath{ {./images/} }

\allowdisplaybreaks

 \usepackage{graphicx,fancybox,latexsym,epsfig}
\usepackage{fancyhdr,amsmath,times,amsxtra,amssymb}
\usepackage{color}

\newtheorem{assumption}{Assumption}

\def\Pr{\mathop{\rm Pr}}

\def\argmin{\mathop{\rm arg\, min}}

\def\B{{\mathcal B}}

\def\P{{\mathcal P}}

\def\sPr{{\mathsf{Pr}}}

\def\sX{{\mathds X}}

\def\sU{{\mathds U}}

\def\sU{{\mathds U}}

\newcommand{\R}{\mathds{R}}
\newcommand{\Zplus}{\mathbb{Z}_+}
\newcommand{\N}{\mathbb{N}}

\newcommand{\dd}{\mathrm{d}}

\allowdisplaybreaks

\ShortHeadings{Near Optimality of Finite Memory Feedback Policies in POMPDs}{Kara and Y\"uksel}
\firstpageno{1}

\begin{document}

\title{Near Optimality of Finite Memory Feedback Policies in Partially Observed Markov Decision Processes
}

\author{\name Ali Devran Kara \email alikara@umich.edu \\
       \addr Department of Mathematics\\
       University of Michigan\\
       Ann Arbor, MI 48109-1043, USA
       \AND
       \name  Serdar Y\"uksel \email yuksel@queensu.ca \\
       \addr  Department of Mathematics and Statistics\\
       Queen's University\\
      Kingston, ON, Canada,}

\editor{-}

\maketitle

\begin{abstract}%
In the theory of Partially Observed Markov Decision Processes (POMDPs), existence of optimal policies have in general been established via converting the original partially observed stochastic control problem to a fully observed one on the belief space, leading to a belief-MDP. However, computing an optimal policy for this fully observed model, and so for the original POMDP, using classical dynamic or linear programming methods is challenging even if the original system has finite state and action spaces, since the state space of the fully observed belief-MDP model is always uncountable. Furthermore, there exist very few rigorous value function approximation and optimal policy approximation results, as regularity conditions needed often require a tedious study involving the spaces of probability measures leading to properties such as Feller continuity. In this paper, we study a planning problem for POMDPs where the system dynamics and measurement channel model are assumed to be known. We construct an approximate belief model by discretizing the belief space using only finite window information variables. We then find optimal policies for the approximate model and we rigorously establish near optimality of the constructed finite window control policies in POMDPs under mild non-linear filter stability conditions and the assumption that the measurement and action sets are finite (and the state space is real vector valued). We also establish a rate of convergence result which relates the finite window memory size and the approximation error bound, where the rate of convergence is exponential under explicit and testable exponential filter stability conditions. While there exist many experimental results and few rigorous asymptotic convergence results, an explicit rate of convergence result is new in the literature, to our knowledge.
 \end{abstract}


\section{Introduction}
For Partially Observed Stochastic Control, also known as Partially Observed Markov Decision Problems (POMDPs), existence of optimal policies have in general been established via converting the original partially observed stochastic control problem to a fully observed Markov Decision Problem (MDP) one on the belief space, leading to a belief-MDP. However, computing an optimal policy for this fully observed model, and so for the original POMDP, using classical methods (such as dynamic programming, policy iteration, linear programming) is challenging even if the original system has finite state and action spaces, since the state space of the fully observed model is always uncountable.

In the MDP theory, various methods have been developed to compute approximately optimal policies by reducing the original problem into a simpler one. A partial list of these techniques is as follows: approximate dynamic programming, approximate value or policy iteration, simulation-based techniques, neuro-dynamic programming (or reinforcement learning), state aggregation, etc. \citep{DuPr12,Ber75,chow1991optimal}. \citet{SaYuLi15c} investigated finite action and state approximations of fully observed stochastic control problems with general state and action spaces under the discounted cost and average cost optimality criteria where weak continuity conditions were shown to be sufficient for near optimality of finite state approximations  in the sense that optimal policies obtained from these models asymptotically achieve the optimal cost for the original problem under the weak continuity assumption on the controlled transition kernel. 




On POMDPs, however, the problem of approximation is significantly more challenging. Most of the studies in the literature are algorithmic and computational contributions  \citep{porta2006point,ZhHa01}. These studies develop computational algorithms, utilizing structural convexity/concavity properties of the value function under the discounted cost criterion. \citet{spaan2005perseus} provide an insightful algorithm which may be regarded as a quantization of the belief space; however, no rigorous convergence results are provided. \citet{smith2012point,pineau2006anytime} also present quantization based algorithms for the belief state, where the state, measurement, and the action sets are finite. \citet{zhang2014covering} also provides a computationally efficient approximation scheme by quantizing the belief space uniformly under $L_1$ distance.


For partially observed setups, \citet{SYLTAC2017POMDP,SaYuLi15c} introduce a rigorous approximation analysis after establishing weak continuity conditions on the transition kernel defining the (belief-MDP) via the non-linear filter \citep{FeKaZa12, KSYWeakFellerSysCont}, and show that finite model approximations obtained through quantization are asymptotically optimal and the control policies obtained from the finite model can be applied to the actual system with asymptotically vanishing error as the number of quantization bins increases. Another rigorous set of studies is by \citet{zhou2008density,zhou2010solving} where the authors provide an explicit quantization method for the set of probability measures containing the belief states, where the state space is parametrically representable under strong density regularity conditions. The quantization is done through the approximations as measured by Kullback-Leibler divergence (relative entropy) between probability density functions. Further recent studies include \citet{mao2020,Mahajan2019}. \citet{Mahajan2019} present a notion of approximate information variable and studies near optimality of policies that satisfies the approximate information state property. \citet{mao2020} analyzes a similar problem under a decentralized setup. Our explicit approximation results in this paper will find applications in both of these studies. 

We refer the reader to the survey papers by \citet{Lov91-(b),Whi91,hansen2013solving} and the recent book by \citet{Kri16} for further structural results as well as algorithmic and computational methods for approximating POMDPs. Notably, for POMDPs \citet{Kri16} presents structural results on optimal policies under monotonicity conditions of the value function in the belief variable.

For our work, we specifically focus on finite memory approximations. With regard to approximations based on finite memory, the following two papers are particularly relevant to our paper:

\citet{yu2008near} study near optimality of finite window policies for average cost problems where the state, action and observation spaces are finite; under the condition that the liminf and limsup of the average cost are equal and independent of the initial state, the paper establishes the near-optimality of (non-stationary) finite memory policies. Here, a concavity argument building on  a work of \citet{Feinberg2} (which becomes consequential by the equality assumption) and the finiteness of the state space is crucial. The paper shows that for any given $\epsilon>0$, there exists an $\epsilon$-optimal finite window policy. However, the authors do not provide a performance bound related to the length of the window, and in fact the proof method builds on convex analysis. Nonetheless, the constant property of the value functions over initial priors is related to unique ergodicity, and thus the stability problem of non-linear filters, which is a topic of current investigation particularly in the controlled setup. 

In another related direction, \citet{white1994finite} study finite memory approximation techniques for POMDPs with finite state, action and measurements. The POMDP is reduced to a belief MDP and the worst and best case predictors prior to the $N$ most recent information variables are considered to build an approximate belief MDP. The original value function is bounded using these approximate belief MDPs that use only finite memory, where the finiteness of the state space is critically used. Furthermore, a loss bound is provided for a suboptimally constructed policy that only uses finite history, where the bound depends on a more specific ergodicity coefficient (which requires restrictive, sample pathwise, contraction properties). In our paper, we consider more general signal spaces and more relaxed filter stability conditions, and establish explicit rates of convergence results. We also rigorously establish the relation of the loss bound to nonlinear filter stability and state space reduction techniques for MDPs.   

A recent work by the authors \citet{kara2021convergence} introduces a different finite history approximation technique where the approximation is done via an alternative belief MDP-reduction method rather than direct discretization of the space of probability measures. It is shown that the approximation error can be related to the controlled filter stability in terms of the total variation distance, whereas in this paper, the error bound is also shown to be related to more general, and in particular weak convergence inducing, metrics. Although, the approximation technique introduced in \citet{kara2021convergence} provides an error upper bound in terms of the more stringent total variation distance, it proves to be numerically efficient as it is shown that a finite history Q learning algorithm converges to the optimality equation of an approximate model. Our analysis here requires less stringent conditions on filter stability, however the use of the bounded-Lipschitz metric on probability measures leads to a significantly more tedious analysis.

{\bf Contributions.} In this paper, we rigorously establish near optimality of finite memory feedback control policies for the case where the actions and measurements are finite (with the state being real vector valued), provided that the controlled non-linear filter is stable in a sense to be presented in the paper. We also explicitly relate the approximation error with the window size. This is the first rigorous result, to our knowledge, where finite window policies are shown to be $\epsilon$-optimal with an explicit rate of convergence with respect to the window size. 


\subsection{Preliminaries and the Main Results}\label{prb}
Let $\mathds{X} \subset \mathds{R}^m$ denote a Borel set which is the state space of a partially observed controlled Markov process. Here and throughout the paper $\Zplus$ denotes the set of non-negative
integers and $\mathds{N}$ denotes the set of positive integers. Let
$\mathds{Y}$ be a finite set denoting the observation space of the model, and let the state be observed through an
observation channel $Q$. The observation channel, $Q$, is defined as a stochastic kernel (regular
conditional probability) from  $\mathds{X}$ to $\mathds{Y}$, such that
$Q(\,\cdot\,|x)$ is a probability measure on the power set $P(\mathds{Y})$ of $\mathds{Y}$ for every $x
\in \mathds{X}$, and $Q(A|\,\cdot\,): \mathds{X}\to [0,1]$ is a Borel
measurable function for every $A \in P(\mathds{Y})$.  A
decision maker (DM) is located at the output of the channel $Q$, and hence it only sees the observations $\{Y_t,\, t\in \Zplus\}$ and chooses its actions from $\mathds{U}$, the action space which is a finite
subset of some Euclidean space. An {\em admissible policy} $\gamma$ is a
sequence of control functions $\{\gamma_t,\, t\in \Zplus\}$ such
that $\gamma_t$ is measurable with respect to the $\sigma$-algebra
generated by the information variables
$
I_t=\{Y_{[0,t]},U_{[0,t-1]}\}, \quad t \in \mathds{N}, \quad
  \quad I_0=\{Y_0\},
$
where
\begin{equation}
\label{eq_control}
U_t=\gamma_t(I_t),\quad t\in \Zplus,
\end{equation}
are the $\mathds{U}$-valued control
actions and 
$Y_{[0,t]} = \{Y_s,\, 0 \leq s \leq t \}, \quad U_{[0,t-1]} =
  \{U_s, \, 0 \leq s \leq t-1 \}.$

\noindent We define $\Gamma$ to be the set of all such admissible policies. The update rules of the system are determined by (\ref{eq_control}) and the following
relationships:
\[  \Pr\bigl( (X_0,Y_0)\in B \bigr) =  \int_B \mu(dx_0)Q(dy_0|x_0), \quad B\in \mathcal{B}(\mathds{X}\times\mathds{Y}), \]
where $\mu$ is the (prior) distribution of the initial state $X_0$, and
\begin{eqnarray*}
\label{eq_evol}
 \Pr\biggl( (X_t,Y_t)\in B \, \bigg|\, (X,Y,U)_{[0,t-1]}=(x,y,u)_{[0,t-1]} \biggr)
 = \int_B \mathcal{T}(dx_t|x_{t-1}, u_{t-1})Q(dy_t|x_t),  
\end{eqnarray*}
$B\in \mathcal{B}(\mathds{X}\times\mathds{Y}), t\in \mathds{N},$ where $\mathcal{T}$ is the transition kernel of the model which is a stochastic kernel from $\mathds{X}\times
\mathds{U}$ to $\mathds{X}$. Note that, although $\mathds{Y}$ is finite, we use integral sign instead of the summation sign for notation convenience by letting the measure to be sum of dirac-delta measures.  We let the objective of the agent (decision maker) be the minimization of the infinite horizon discounted cost, 
  \begin{align*}
    J_{\beta}(\mu,{\cal T},\gamma)= E_\mu^{{\cal T},\gamma}\left[\sum_{t=0}^{\infty} \beta^t c(X_t,U_t)\right]
  \end{align*}
 \noindent for some discount factor $\beta \in (0,1)$, over the set of admissible policies $\gamma\in\Gamma$, where $c:\mathds{X}\times\mathds{U}\to\R$ is a Borel-measurable stage-wise cost function and $E_\mu^{{\cal T},\gamma}$ denotes the expectation with initial state probability measure $\mu$ and transition kernel ${\cal T}$ under policy $\gamma$. Note that $\mu\in\mathcal{P}(\mathds{X})$, where we let $\mathcal{P}(\mathds{X})$ denote the set of probability measures on $\mathds{X}$. We define the optimal cost for the discounted infinite horizon setup as a function of the priors and the transition kernels as
\begin{align*}
  J_{\beta}^*(\mu,{\cal T})&=\inf_{\gamma\in\Gamma} J_{\beta}(\mu,{\cal T},\gamma).
\end{align*}
For partially observed stochastic problems, the optimal policies use all the available information in general. The question we ask is the following one: suppose we define an {\em  $N$-memory admissible policy} $\gamma^N$ so that $\gamma^N$ is a
sequence of control functions $\{\gamma_t,\, t\in \Zplus\}$ such
that $\gamma_t$ is measurable with respect to the $\sigma$-algebra
generated by the information variables
\begin{align}\label{new_info}
I_t^N&=\{Y_{[t-N,t]},U_{[t-N,t-1]}\}, \text{ if } t\geq N,\nonumber\\ 
I_t^N&=\{Y_{[0,t]},U_{[0,t-1]}\}, \text{ if } 0< t<N,\nonumber\\
  I_0&=\{Y_0\},
\end{align}
that is the controller can only have access to the information variables through a window whose length is $N$. 
We define $\Gamma^N$ to be the set of all such $N$-memory admissible policies. Similarly, we define the optimal cost function under $N$-memory admissible policies as
\begin{align*}
&J_\beta^N(\mu,\mathcal{T})=\inf_{\gamma^N\in\Gamma^N}J_\beta(\mu,\mathcal{T},\gamma^N).
\end{align*}

Under this setup, we will study the following problem.
\begin{itemize}
\item [{\it Problem:}] Under suitable conditions, can we find explicit bounds on $J_\beta^N(\mu,\mathcal{T})-J_\beta^*(\mu,\mathcal{T})$ in terms of $N$ and a constructive approximate solution achieving this bound?
\end{itemize}

Our goal is to find the best possible control policy in $\Gamma^N$ that is, in the set of policies that  use only a finite history of information variables, in an offline setting by reducing the problem to a simpler approximate setup where we assume that the system dynamics are known to the designer. A general summary of the approach we will follow to answer this problem is as follows: We first define the belief MDP counterpart of the partially observed system. Then, we construct a finite subset of the belief state space using the probability distributions that can be achieved using the finite window information variables ($I_t^N$'s) from a fixed probability distribution. This finite subset leads to an approximate MDP model for which we find optimal policies. The calculation of the policies is greatly simplified compared to the calculation of optimal policies for the original POMDP model. Finally, we show that the loss occurring from applying this approximate policy to the original model can be upper bounded by the expected error of the dicretization of the belief space. The loss is evaluated compared to the best possible \textit{admissible} policy in the set $\Gamma$. The accumulating error then, can be represented in relation to the filter stability problem, that is, how fast the controlled process forgets its initial distribution as it observes the information variables from the system.

Note that although we take the infimum over all $N$-memory admissible policies, we will explicitly construct finite window policies which will be {\it time-invariant}, or a {\it finite-state probabilistic automaton} \citep[as is also referred to by][]{yu2008near} that accepts as inputs the finite window of observations and actions, and produces as outputs the control actions in a time-invariant/stationary fashion. Accordingly, the infimum above for $J_\beta^N(\mu,\mathcal{T})$ can be replaced with the minimum over such policies.

We answer the problem above affirmatively in Theorems \ref{mainThmNIPS}, \ref{mainThmNIPS2}, and \ref{unif_cont} under complementary conditions. 

\section{Regularity and Stability Properties of the Belief-MDP}
In this section, we introduce the belief MDP reduction of POMDPs and provide regularity properties of the belief MDPs.

\subsection{Convergence Notions for Probability Measures}
For the analysis of the technical results, we will use different notions of convergence for sequences of probability measures.

Two important notions of convergences for sequences of probability measures are weak convergence, and convergence under total variation. For some $N\in\N$ a sequence $\{\mu_n,n\in\N\}$ in $\mathcal{P}(\R^N)$ is said to converge to $\mu\in\mathcal{P}(\R^N)$ \emph{weakly} if $\int_{\R^N}c(x)\mu_n(dx) \to \int_{\R^N}c(x)\mu(dx)$ for every continuous and bounded $c:\R^N \to \R$.
One important property of weak convergence is that the space of probability measures on a complete, separable, and metric (Polish) space endowed with the topology of weak convergence is itself complete, separable, and metric \citep{Par67}. One such metric is the bounded Lipschitz metric  \cite[p.109]{villani2008optimal}, which is defined for $\mu,\nu \in \P(\mathds{X})$ as 
\begin{equation}\label{BLmetric}
\rho_{BL}(\mu,\nu):=\sup_{\|f\|_{BL}\leq1} | \int f d\mu - \int f d\nu | 
\end{equation}
where \[ \|f\|_{BL}:=\|f\|_\infty+\sup_{x\neq y}\frac{|f(x)-f(y)|}{d(x,y)} \]
and $\|f\|_\infty=\sup_{x\in\mathds{X}}|f(x)|$.

  For probability measures $\mu,\nu \in \mathcal{P}(\R^N)$, the \emph{total variation} metric is given by
  \begin{align*}
    \|\mu-\nu\|_{TV}&=2\sup_{B\in\mathcal{B}(\R^N)}|\mu(B)-\nu(B)|=\sup_{f:\|f\|_\infty \leq 1}\left|\int f(x)\mu(\dd x)-\int f(x)\nu(\dd x)\right|,
  \end{align*}
  \noindent where the supremum is taken over all measurable real $f$ such that $\|f\|_\infty=\sup_{x\in\R^N}|f(x)|\leq 1$. A sequence $\mu_n$ is said to converge in total variation to $\mu \in \mathcal{P}(\R^N)$ if $\|\mu_n-\mu\|_{TV}\to 0$.

\subsection{Ergodicity and Filter Stability Properties of Partially Observed MDPs}


Given a prior $\mu \in \P(\mathds{X})$ and a policy $\gamma \in \Gamma$, we define the filter and
predictor for a POMDP in the following.
\begin{definition}
The one step predictor process is defined as the sequence of conditional probability measures 
\begin{align*}
\pi_{n-}^{\mu,\gamma}(\cdot)&=P^{\mu,\gamma}(X_n\in \cdot|Y_{[0,n-1]},U_{[0,n-1]}=\gamma_n(Y_{[0,n-1]},U_{[0,n-2]}))=P^{\mu,\gamma}(X_n\in \cdot|Y_{[0,n-1]}),\quad n\in \mathbb{N} 
\end{align*}
where $P^{\mu,\gamma}$ is the probability measure induced by the prior $\mu$ and the policy $\gamma$, when $\mu$ is the probability measure on $X_0$.
\end{definition}

\begin{definition}
The filter process is defined as the sequence of conditional probability measures 
\begin{align}\label{filter}
\pi_{n}^{\mu,\gamma}(\cdot)&=P^{\mu,\gamma}(X_n\in \cdot|Y_{[0,n]},U_{[0,n-1]}=\gamma_n(Y_{[0,n-1]},U_{[0,n-2]}))=P^{\mu,\gamma}(X_n\in \cdot|Y_{[0,n]}),\quad  n \in \mathbb{N} 
\end{align}
where $P^{\mu,\gamma}$ is the probability measure induced by the prior $\mu$ and the policy $\gamma$.
\end{definition}



%

\begin{definition}\cite[Equation 1.16]{dobrushin1956central}\label{dob_def}
For a kernel operator $K:S_{1} \to \mathcal{P}(S_{2})$ (that is a regular conditional probability from $S_1$ to $S_2$) for standard Borel spaces $S_1, S_2$, we define the Dobrushin coefficient as:
\begin{align}
\delta(K)&=\inf\sum_{i=1}^{n}\min(K(x,A_{i}),K(y,A_{i}))\label{Dob_def}
\end{align}
where the infimum is over all $x,y \in S_{1}$ and all partitions $\{A_{i}\}_{i=1}^{n}$ of $S_{2}$.
\end{definition}
We note that this definition holds for continuous or finite/countable spaces $S_{1}$ and $S_{2}$ and $0\leq \delta(K)\leq 1$ for any kernel operator. 
\begin{example}\label{doub_exmp}
Assume for a finite setup, we have the following stochastic transition matrix
\begin{equation*}
K = 
\begin{pmatrix}
\frac{1}{3} & \frac{1}{3}  & \frac{1}{3} \\
0 & \frac{1}{2} & \frac{1}{2} \\
\frac{3}{4} & 0 & \frac{1}{4}
\end{pmatrix}
\end{equation*}
The Dobrushin coefficient is the minimum over any two rows where we sum the minimum elements among those rows. For this example, the first and the second rows give $\frac{2}{3}$, the first and the third rows give $\frac{7}{12}$ and the second and the third rows give $\frac{1}{4}$. Then the  Dobrushin coefficient is $\frac{1}{4}$.
\end{example}

Let
\[ \tilde{\delta}({\cal T}):=\inf_{u \in \mathds{U}} \delta({\cal T}(\cdot|\cdot,u)). \]
%

\begin{definition}
For $\mathds{X}\subset \mathds{R}^m$ for some $m\in\mathds{N}$, and for two probability measures $\mu,\nu \in \P(\mathds{X})$, $\mu$ is said to be absolutely continuous with respect to $\nu$ if $\mu(A) = 0$ for every set $A\in\B(\mathds{X})$ for which $\nu(A) = 0$. We denote the absolute continuity of $\mu$ with respect to $\nu$ by $\mu\ll\nu$.
\end{definition}

\begin{theorem} \cite[Theorem 3.3]{mcdonald2020exponential} \label{curtis_result}
Assume that for $\mu,\nu \in \P(\mathds{X})$, we have $\mu\ll\nu$. Then we have (exponential filter stability)
\begin{align*}
E^{\mu,\gamma}\left[\|\pi_{n+1}^{\mu,\gamma}-\pi_{n+1}^{\nu,\gamma}\|_{TV}\right]\leq (1-\tilde{\delta}(\mathcal{T}))(2-\delta(Q))E^{\mu,\gamma}\left[\|\pi_{n}^{\mu,\gamma}-\pi_{n}^{\nu,\gamma}\|_{TV}\right].
\end{align*}
In particular, defining $\alpha:=(1-\tilde{\delta}(\mathcal{T}))(2-\delta(Q))$, we have
\begin{align*}
E^{\mu,\gamma}\left[\|\pi_{n}^{\mu,\gamma}-\pi_{n}^{\nu,\gamma}\|_{TV}\right]\leq 2\alpha^n.
\end{align*}
\end{theorem}

This result will be a key ingredient for our main results. It provides conditions on when the belief-state processes for a given POMDP under different priors get closer when they are fed with same observation processes, in expectation under the true probability space. In a vague sense, if the state process is tracked only using a finite window of recent measurement and control variables (and forgets the past observations and actions), then the amount of mismatch from the true filter can be bounded with an error that is exponentially diminishing with the window size. The relationship is via the term $\sup_{\pi\in\P(\mathds{X})}\sup_{\gamma\in\Gamma}E_\pi^{\gamma}\left[\rho_{BL}\left(P^{\pi}(X_N\in\cdot|Y_{[0,N]}),P^{\hat{\pi}}(X_N\in\cdot|Y_{[0,N]})\right)\right]$ which appears crucially in (\ref{filterSEqnSum}). We note that if one does not wish to have an explicit rate of convergence result, one could have more relaxed conditions for filter stability which will still lead to rigorous approximation results on the performance of finite window policies via the controlled filter stability analysis in \cite{MYRobustControlledFS}.

\subsection{Reduction to Fully Observed Models and Regularity Properties of Belief-MDPs}
It is by now a standard result that, for optimality analysis, any POMDP can be reduced to a completely observable Markov decision process \citep{Yus76,Rhe74}, whose states are the posterior state distributions or {\it beliefs} of the observer or the filter process as defined in (\ref{filter}); that is, the state at time $n$ is
\begin{align}
\sPr\{X_{n} \in \,\cdot\, | Y_0,\ldots,Y_n, U_0, \ldots, U_{n-1}\} \in \P(\sX). \nonumber
\end{align}
We call this equivalent process the filter process \index{Belief-MDP}. The filter process has state space $\mathcal{Z} = \P(\sX)$  and action space $\sU$. Here, $\mathcal{Z}$ is equipped with the Borel $\sigma$-algebra generated by the topology of weak convergence \citep{Bil99}. As noted earlier, under this topology, $\mathcal{Z}$ is a standard Borel space \citep{Par67}. 
Then, the transition probability $\eta$ of the filter process can be constructed as follows \citep[see also][]{Her89}. If we define the measurable function 
\[F(z,u,y) := F(\,\cdot\,|y,u,z) = \Pr\{X_{n+1} \in \,\cdot\, | Z_n = z, U_n = u, Y_{n+1} = y\}\]
 from ${\cal P}(\mathds{X})\times\mathds{U}\times\mathds{Y}$ to ${\cal P}(\mathds{X})$ and use the stochastic kernel $P(\,\cdot\, | z,u) = \Pr\{Y_{n+1} \in \,\cdot\, | Z_n = z, U_n = u\}$ from ${\cal P}(\mathds{X})\times\mathds{U}$ to $\mathds{Y}$, we can write $\eta$ as
\begin{align}
\eta(\,\cdot\,|z,u) = \int_{\mathds{Y}} 1_{\{F(z,u,y) \in \,\cdot\,\}} P(dy|z,u). \label{beliefK}
\end{align}

The one-stage cost function $\tilde{c}:{\cal P}(\mathds{X}) \times \mathds{U}\rightarrow[0,\infty)$ of the filter process is given by 
\begin{align}\label{belief_cost}
\tilde{c}(z,u) := \int_{\sX} c(x,u) z(dx),
\end{align}
which is a Borel measurable function. Hence, the filter process is a completely observable Markov process with the components $(\mathcal{Z},\sU,\tilde{c},\eta)$.

For the filter process, the information variables is defined as
\[
\tilde{I}_t=\{Z_{[0,t]},U_{[0,t-1]}\}, \quad t \in \mathds{N}, \quad
  \quad \tilde{I}_0=\{Z_0\}.
\]

It is well known that an optimal control policy of the original POMDP can use the belief $Z_t$ as a sufficient statistic for optimal policies \citep[see][]{Yus76,Rhe74}, provided they exist. More precisely, the filter process is equivalent to the original POMP  in the sense that for any optimal policy for the filter process, one can construct a policy for the original POMP which is optimal. On existence, we note the following. 

With the recent results by \citet{FeKaZg14,KSYWeakFellerSysCont} the transition model of the belief-MDP can be shown to satisfy weak continuity conditions on the belief state and action variables, and accordingly we have that the measurable selection conditions \citep[Chapter 3]{HernandezLermaMCP} apply.  Notably, we state the following.

\begin{assumption}\label{TV_channel}
\begin{itemize}
\item[(i)] The transition probability $\mathcal{T}(\cdot|x,u)$ is weakly continuous in $(x,u)$, i.e., for any $(x_n,u_n)\to (x,u)$, $\mathcal{T}(\cdot|x_n,u_n)\to \mathcal{T}(\cdot|x,u)$ weakly.
\item[(ii)] The observation channel $Q(\cdot|x,u)$ is continuous in total variation, i.e., for any $(x_n,u_n) \to (x,u)$, $Q(\cdot|x_n,u_n) \rightarrow Q(\cdot|x,u)$ in total variation.
\end{itemize}
\end{assumption}

\begin{assumption}\label{TV_kernel}
\begin{itemize}
\item[(i)] The transition probability $\mathcal{T}(\cdot|x,u)$ is continuous in total variation in $(x,u)$, i.e., for any $(x_n,u_n)\to (x,u)$, $\mathcal{T}(\cdot|x_n,u_n) \to \mathcal{T}(\cdot|x,u)$ in total variation.
\item[(ii)] The observation channel $Q(\cdot|x)$ is independent of the control variable.
\end{itemize}
\end{assumption}

\begin{theorem} 
\begin{itemize}
\item[(i)] \citep{FeKaZg14} \label{TV_channel_thm}
Under Assumption \ref{TV_channel}, the transition probability $\eta(\cdot|z,u)$ of the filter process is weakly continuous in $(z,u)$.
\item[(ii)] \citep{KSYWeakFellerSysCont} \label{TV_kernel_thm}
Under Assumption \ref{TV_kernel}, the transition probability $\eta(\cdot|z,u)$ of the filter process is weakly continuous in $(z,u)$.
\end{itemize}
\end{theorem}

Under the above weak continuity conditions, the measurable selection conditions \citep[Chapter 3]{HernandezLermaMCP} apply and a solution to the discounted cost optimality equation exists, and accordingly an optimal control policy exists. This policy is stationary (in the belief state). Thus there exists a
function $\Phi : \P(\mathds{X}) \to \mathds{U}$ such that for any policy $\gamma$ for a prior $\mu$
\begin{align*}
\gamma(y_{[0,n]})=\Phi\left(P^{\mu,\gamma}(X_n\in\cdot|Y_{[0,n]}=y_{[0,n]})\right)=\Phi(\pi_n^{\mu,\gamma})
\end{align*}

In particular, we have that $J_\beta^*(\mu,\mathcal{T},Q)=J_\beta^*(\mu,\eta)$. This will be the case in our paper, under the assumptions we will work with.

For the rest of the paper, we will use $\gamma$ for the belief process policy $\Phi$ for consistency of notation.

The following supporting result, to be used later in the paper, provides further regularity properties on the transition model for the belief model under mild conditions on the fully observed model. This result may be useful for POMDP theory beyond the application considered in this paper. We also note that the bounds in (i) and (iii) below are applicable when we only have filter stability, but not exponential filter stability \citep{MYRobustControlledFS}, whereas items (ii)-(iv) will be used under exponential filter stability in this paper. 
\begin{theorem}\label{belief_kernel_regularity}
\begin{itemize}
\item [i.] Assume that 
\begin{align*}
\|\mathcal{T}(\cdot|x,u)-\mathcal{T}(\cdot|x',u)\|_{TV}\leq \alpha_{\mathds{X}}|x-x'|
\end{align*}
for some $\alpha_{\mathds{X}}<\infty$ for all $u\in\mathds{U}$. We have
\begin{align*}
\rho_{BL}\left(\eta(\cdot|z,u),\eta(\cdot|z',u)\right)\leq 3(1+\alpha_{\mathds{X}})\rho_{BL}(z,z').
\end{align*}

\item[ii.]
 Assume that 
\begin{align*}
\|\mathcal{T}(\cdot|x,u)-\mathcal{T}(\cdot|x',u)\|_{TV}\leq \alpha_{\mathds{X}}|x-x'|
\end{align*}
for some $\alpha_{\mathds{X}}<\infty$ for all $u\in\mathds{U}$. Then, under the conditions of Theorem \ref{curtis_result} we have
\begin{align*}
\rho_{BL}\left(\eta(\cdot|z,u),\eta(\cdot|z',u)\right)\leq (3-2\delta(Q))(1+\alpha_{\mathds{X}})\rho_{BL}(z,z').
\end{align*}

\item[iii.] Without any assumption
\begin{align*}
\rho_{BL}\left(\eta(\cdot|z,u),\eta(\cdot|z',u)\right)\leq 3\|z-z'\|_{TV}.
\end{align*}
\item[iv.] Under the conditions of Theorem \ref{curtis_result}
\begin{align*}
\rho_{BL}\left(\eta(\cdot|z,u),\eta(\cdot|z',u)\right)\leq (3-2\delta(Q))(1-\tilde{\delta}(\mathcal{T}))\|z-z'\|_{TV}.
\end{align*}
\end{itemize}
\end{theorem}

\begin{proof} See Section \ref{SecProofReg}. \end{proof}

%
\section{Approximate Model Construction: Finite Belief-MDP through Finite Memory}\label{finite_state}
In this section, we will construct a finite state space by quantizing the belief state space so that the approximate finite model is obtained using only a finite memory. 

Our construction builds on but significantly differs from the approach by \citet{naci_abi_book,SaYuLi15c}. As we will explain, we cannot afford to use uniform quantization in our setup, which was a crucial tool used by  \citet{naci_abi_book,SaYuLi15c}. 

As we discussed in the previous section, we can write the infinite horizon cost as
\begin{align*}
J_\beta(\mathcal{T},Q,\gamma,\mu) &=\sum_{t=0}^{\infty}\beta^tE_\mu\left[\tilde{c}\left(\pi_t,\gamma(\pi_t)\right)\right] \\
& =\sum_{t=0}^{N-1}\beta^tE_\mu\left[\tilde{c}\left(\pi_t,\gamma(\pi_t)\right)\right]+\sum_{t=N}^{\infty}\beta^tE_{\mu}\left[\tilde{c}\left(\pi_t,\gamma(\pi_t)\right)\right].
\end{align*}


Now we focus on the second term:
\begin{align*}
&\sum_{t=N}^{\infty}\beta^tE_{\mu}\left[\tilde{c}\left(\pi_t,\gamma(\pi_t)\right)\right] \\
& =\sum_{t=N}^{\infty}\beta^tE_{\mu}\left[\tilde{c}\left(P^{\mu,\gamma}(X_t\in \cdot|Y_{[0,t]},U_{[0,t-1]}),\gamma(P^{\mu,\gamma}(X_t\in \cdot|Y_{[0,t]},U_{[0,t-1]}))\right)\right].\nonumber
\end{align*}

Notice that for any time step $t\geq N$ and for a fixed observation realization sequence $y_{[0,t]}$ and for a fixed control action sequence $u_{[0,t-1]}$, the state process can be viewed as
\begin{align*}
&P^{\mu}(X_t\in \cdot|Y_{[0,t]}=y_{[0,t]},U_{[0,t-1]}=u_{[0,t-1]})\\
&=P^{\pi_{{t-N}_-}}(X_t\in\cdot|Y_{[t-N,t]}=y_{[t-N,t]},U_{[t-N,t-1]}=u_{[t-N,t-1]})
\end{align*}
where 
\begin{align*}
\pi_{{t-N}_-}(\cdot)=P^{\mu}(X_{t-N}\in \cdot|Y_{[0,t-N-1]}=y_{[0,t-N-1]},U_{[0,t-N-1]}=u_{[0,t-N-1]}).
\end{align*}
That is, we can view the state as the Bayesian update of $\pi_{t-N_-}$, the predictor at time $t-N$, using the observations $Y_{t-N},\dots,Y_{t}$. Notice that with this representation only the most recent $N$ observation realizations are used for the update and the past information of the observations is embedded in $\pi_{t-N_-}$. 

Hence, we can view the state space for time stages $t\geq N$ as
\begin{align*}
\mathcal{Z}=\left\{P^{\pi}(X_N\in\cdot|Y_{[0,N]},U_{[0,N-1]}); \pi\in\P(\mathds{X}), Y_{[0,N]}\in \mathds{Y}^{N+1}, U_{[0,N-1]}\in \mathds{U}^{N}\right\}
\end{align*}

Consider the following finite set $\mathcal{Z}_{\hat{\pi}}^N$ defined by, for a fixed probability measure $\hat{\pi}\in\P(\mathds{X})$, 
\begin{align*}
\mathcal{Z}_{\hat{\pi}}^N:=\left\{P^{\hat{\pi}}(X_N\in\cdot|Y_{[0,N]},U_{[0,N-1]}); Y_{[0,N]}\in \mathds{Y}^{N+1}, U_{[0,N-1]}\in \mathds{U}^{N}\right\}.
\end{align*}

Define the map $F: \mathcal{Z} \to \mathcal{Z}_{\hat{\pi}}^N$, which will serve as the quantizer, with
\begin{align}\label{quant_map}
F(z):=\argmin_{y \in \mathcal{Z}_{\hat{\pi}}^N}\rho_{BL}(z,y) 
\end{align}

This map separates the set $\mathcal{Z}$ into $|\mathds{Y}|^{N+1}\times|\mathds{U}|^N$ sets, accordingly this quantizes the set of probability measures.

To complete our approximate controlled Markov model, we now define the one-stage cost function $c^N:\mathcal{Z}_{\hat{\pi}}^N\times \mathds{U}\to [0,\infty)$ and the transition probability $\eta^N$ on $\mathcal{Z}_{\hat{\pi}}^N$ given realizations in $\mathcal{Z}_{\hat{\pi}}^N\times \mathds{U}$: For a given $z_i=P^{\hat{\pi}}(X_N\in\cdot|y^i_{[0,N]},u^i_{[0,N-1]})$ and control action $u$
\begin{align*}
c^N(z_i,u)&=c^N(P^{\hat{\pi}}(X_N\in\cdot|y^i_{[0,N]},u^i_{[0,N-1]}),u):=\tilde{c}(P^{\hat{\pi}}(X_N\in\cdot|y^i_{[0,N]},u^i_{[0,N-1]}),u),
\end{align*} 
where $\tilde{c}$ is defined in (\ref{belief_cost}) and
\begin{align*}
\eta^N(\cdot|z_i,u)&=\eta_N(\cdot|P^{\hat{\pi}}(X_N\in\cdot|y^i_{[0,N]},u^i_{[0,N-1]},u)) \\
& \qquad \qquad :=F\ast \eta(\cdot|P^{\hat{\pi}}(X_N\in\cdot|y^i_{[0,N]},u^i_{[0,N-1]},u),
\end{align*} 
where 
\[ F\ast \eta(z'|z,u)=\eta(\{z\in {\cal P}(\mathds{X}) : F(z)=z'\}|z,u),\]
for all $z' \in \mathcal{Z}_{\hat{\pi}}^N$. 

We thus have defined a finite state MDP with the state space $\mathcal{Z}_{\hat{\pi}}^N$, action space $\mathds{U}$, cost function $c_N$ and the transition probability $\eta_N$. 

An optimal policy $\gamma_N^*$ for this finite state model is a function taking values from the finite state space and hence at some time step $t\geq N$ it only uses $N$ most recent observations and control action variables that is, $\gamma_N^*$ is a measurable function of the information set $I_t^N$ defined in (\ref{new_info}) for all $t\geq N$.

We list the steps to construct the approximate model informally as follows:

\begin{itemize}
\item Fix a probability distribution $\hat{\pi}\in\P(\mathds{X})$ as an estimator for the predictor of $N$ step back ($\pi_{t-N}^-$).
\item Calculate Bayesian updates for all possible realizations $y_{[0,N]},u_{[0,N-1]}$ ($|\mathds{Y}|^{N+1}\times|\mathds{U}|^N$ many realizations) starting from the prior $\hat{\pi}$ to form the finite subset $\mathcal{Z}_{\hat{\pi}}^N$.
\item Calculate $\eta^N$ (approximate transition model) and $c^N$ (approximate cost function) using a nearest neighbor map. Note that from any $z_i\in\mathcal{Z}_{\hat{\pi}}^N$ there are only $|\mathds{Y}|$ many possible transitions under the true dynamics. To construct the $\eta^N$, these possible transitions are mapped to the closest element in $\mathcal{Z}_{\hat{\pi}}^N$.
\item Calculate the value functions and the optimal policies for the finite model with state space $\mathcal{Z}_{\hat{\pi}}^N$, transition kernel $\eta^N$ and the cost function $c^N$.
\end{itemize}

As we have noted before, the complexity of POMDPs in general arises from the structure of the belief state space $\mathcal{Z}$ which is a set of probability measures on $\mathds{X}$. This set is always uncountable and needs to be associated with proper topologies to make the analysis feasible. Approximations for POMDPs are usually done by choosing a finite subset, say $\hat{\mathcal{Z}}$, of the belief state space $\mathcal{Z}$  \citep{smith2012point,pineau2006anytime,SaYuLi15c,zhou2008density,zhou2010solving,Mahajan2019,zhang2014covering}, and finding an approximate MDP model for this finite set. To choose the finite set, the aforementioned works use a uniform quantization scheme, in various topologies on $\mathcal{Z}$. In other words, the quantization is made such that for any $z\in\mathcal{Z}$, there exists an element $\hat{z}\in\hat{\mathcal{Z}}$ with $\|z-\hat{z}\|\leq \epsilon$ for a fixed $\epsilon>0$. The metric to measure distances of the belief states varies for different works, although for finite $\mathds{X}$, $L_1$ distance of distributions is what is used in general for the quantization of $\mathcal{Z}$,  which coincides with total variation and weak convergence topology on $\mathcal{Z}$ when $\mathds{X}$ is finite; for general $\mathds{X}$, a more appropriate and natural topology is the weak convergence topology on $\mathcal{Z}$ which is what we work with in the paper since $\rho_{BL}$ metrizes the weak convergence. 

In this paper, instead of quantizing $\mathcal{Z}$ directly and uniformly, we use finite window information variables ($I_t^N$'s) to construct the finite subset of $\mathcal{Z}$ since our goal is to analyze the effect of the window size on the approximation performance. That is, we use the finite set
\begin{align*}
\mathcal{Z}_{\hat{\pi}}^N:=\left\{P^{\hat{\pi}}(X_N\in\cdot|Y_{[0,N]},U_{[0,N-1]}); Y_{[0,N]}\in \mathds{Y}^{N+1}, U_{[0,N-1]}\in \mathds{U}^{N}\right\}
\end{align*}
constructed using $Y_{[0,N]},U_{[0,N-1]}$. For this set, we cannot afford a uniform discretization scheme. A uniform quantization would mean that for a fixed $\epsilon>0$
\begin{align*}
\rho_{BL}\left(P^{\hat{\pi}}(X_N\in\cdot|y_{[0,N]},u_{[0,N-1]}),P^{\pi}(X_N\in\cdot|y_{[0,N]},u_{[0,N-1]})\right)<\epsilon
\end{align*}
uniformly for any $\pi\in\mathcal{Z}$ and for any $y_{[0,N]}\in \mathds{Y}^{N+1}, u_{[0,N-1]}\in \mathds{U}^{N}$. However, this is in general inapplicable for filter stability problems as it requires for the processes with different starting points to converge uniformly for any realizations of information variables (even for highly unlikely ones). That is why, we follow a different approach and show that we do not have to force uniform quantization, and the error of approximate value functions can be related to the expected error of the form
\begin{align*}
E_\pi^{\gamma}\left[\rho_{BL}\left(P^{\pi}(X_N\in\cdot|Y_{[0,N]},U_{[0,N-1]}),P^{\hat{\pi}}(X_N\in\cdot|Y_{[0,N]},U_{[0,N-1]})\right)\right]
\end{align*}
which in turn can be bounded using Theorem \ref{curtis_result}. Our technical analysis, accordingly, is slightly more tedious at the benefit of arriving at a practical and intuitive finite-memory method whose near optimality is rigorously established.

\begin{remark}\label{obs_quant}
The finite subset $\mathcal{Z}_{\hat{\pi}}^N$ is crucial for the approximation of the belief MDP we construct in this section. This set depends on the choice of $\hat{\pi}$ as the starting probability distribution and the finite window information variables $Y_{[0,N]},U_{[0,N-1]}$. \citet{kara2021convergence} consider a similar approach for construction of an approximate POMDP model by using the same finite set  $\mathcal{Z}_{\hat{\pi}}^N$. However, instead of using a nearest neighbor map for the correspondence between the original belief space and the finite set  $\mathcal{Z}_{\hat{\pi}}^N$, there, the states with matching finite history are used for the correspondence by putting a different topology on the belief space and the set $\mathcal{Z}_{\hat{\pi}}^N$. The method used in this paper naturally results in a smaller approximation error due to the nature of the nearest neighbor map, and lets us work with the weak convergence topology and $\rho_{BL}$ metric. On one hand, \citet{kara2021convergence} provide a greater error term in terms of the total variation distance, which always upper bounds the $\rho_{BL}$ metric, on the other hand, the method used there, is computationally more efficient and the approximate model can actually be learned with a reinforcement learning algorithm that uses finite window information variables.

On the choice of $\hat{\pi}$, we note that Theorem \ref{mainThmNIPS} provides a bound that is independent of the choice of $\hat{\pi}$. However, since this is only an upper bound, it is apparent that the true value of the error may depend on different choices of $\hat{\pi}$. For the approximate model, we use $\hat{\pi}$ as an estimator for the true predictor $\pi_{t_-}^{\mu,\gamma}(\cdot)=P^{\mu,\gamma}(X_t\in \cdot|Y_{[0,t-1]})$ at any given time $t$ under some policy $\gamma$. Since $\hat{\pi}$ is a fixed probability distribution and does not vary with time, we can argue that a reasonable choice would be the time averages of $\pi_{t_-}^{\mu,\gamma^*}(\cdot)$ under an optimal policy $\gamma^*$ and if the hidden state process $X_t$ admits an invariant measure under this optimal policy $\gamma^*$, the time averages of $\pi_{t_-}^{\mu,\gamma^*}(\cdot)$ will converge to the invariant measure of $X_t$ under the optimal policy. However, one issue with this approach is that the designer does not have access to the optimal policy and thus, cannot compute the invariant measure under $\gamma^*$. The designer, instead, can use the approximate optimal policy $\gamma^N$, but the choice of $\hat{\pi}$ also affects the approximate policy $\gamma^N$, and hence, this approach may not feasible for practical purposes. In any case, our result provides a bound uniform over the prior selections.


We now discuss the the role of the realizations of $Y_{[0,N]},U_{[0,N-1]}$. Notice that, to construct $\mathcal{Z}_{\hat{\pi}}^N$, we consider all possible $|\mathds{Y}|^{N+1}\times |\mathds{U}|^{N}$ realizations which reduces the uncountable state space to a finite one but the size of this finite subset grows fast with the size of window, $N$, and with the size of spaces $\mathds{Y}$ and $\mathds{U}$. The bound we provide in Theorem \ref{finite_MDP}, reveals that the approximation error  is related to the expectation over the realizations of $Y_{[0,N]},U_{[0,N-1]}$ under the true dynamics, which suggests that if a realization $y_{[0,N]},u_{[0,N-1]}$, is highly unlikely under the true dynamics, these realizations can be discarded when constructing $\mathcal{Z}_{\hat{\pi}}^N$ for computational ease by reducing the size of $\mathcal{Z}_{\hat{\pi}}^N$. 

In fact this question motivates the following direction. A close look at Definition \ref{dob_def} reveals that the Dobrushin coefficient of a channel and a channel obtained by quantizing the measurements of that same channel are so that the coefficient of the former is a lower bound for the latter. Therefore, one can always quantize the measurement channels further if the original channel satisfies the contraction condition presented for filter stability, with the quantized channel also satisfying the contraction property. This presents a recipe for dealing with both low probability channel outcomes or even continuous space measurement channels; however one would additionally need to show that the value function of the approximate model with quantized measurements would be close to the value function of the original model. This is a possible future direction (using the weak Feller property of the belief-MDP)  \citep[as shown e.g. by][]{SaYuLi15c} and continuity properties of filter updates in measurement realizations \cite{MYRobustControlledFS}.
\end{remark}


\section{Approximation Error Analysis and Rates of Convergence}


An optimal policy for the constructed finite model, $\gamma_N^*$ can be extended to ${\cal P}(\mathds{X})$ and can be used for the original MDP. 

\subsection{Analysis via Expected Filter Approximation Error}
The next result, which is related to the construction given by \citet[Theorem 4.38]{naci_abi_book} \citep[see also][]{SaYuLi15c}, provides a mismatch error for using this policy. This result is going to be the key supporting tool for the main theorem of the paper, which will be presented immediately after. The proof requires multiple technical lemmas and are presented in Section \ref{proofThmMain} (with some supporting but tedious technical steps moved to the Appendix). 

\begin{assumption}\label{belief_reg}
\begin{itemize}
\item  $\rho_{BL}(\eta(\cdot|z,u),\eta(\cdot|z',u))\leq \alpha_{\mathcal{Z}}\rho_{BL}(z,z')$ (see Theorem \ref{belief_kernel_regularity}).
\item  $|\tilde{c}(z,u)-\tilde{c}(z',u)|\leq \alpha_{\tilde{c}}\rho_{BL}(z,z')$  for some $\alpha_{\tilde{c}}$ for all $u\in\mathds{U}$.
\end{itemize}
\end{assumption}

\begin{theorem}\label{finite_MDP}
Under Assumption \ref{belief_reg}, for all $z\in {\cal P}(\mathds{X})$
of the form 
\[z=P^{\pi}(X_N\in\cdot|y_{[0,N]},u_{[0,N-1]}),\]
the following holds if $\beta<\frac{1}{4\alpha_\mathcal{Z}+1}$:

\begin{align*}
&\sup_{\gamma\in\Gamma}E\bigg[J_\beta(\eta,\gamma_N^*,z)-J_\beta(\eta,\gamma^*,z)\bigg|Y_{[0,N]},\gamma(Y_{[0,N-1]})\bigg]\\
&\leq K \sup_{\pi\in\P(\mathds{X})}\sup_{\gamma\in\Gamma}E_\pi^{\gamma}\left[\rho_{BL}\left(P^{\pi}(X_N\in\cdot|Y_{[0,N]},U_{[0,N-1]}),P^{\hat{\pi}}(X_N\in\cdot|Y_{[0,N]},U_{[0,N-1]})\right)\right].
\end{align*}
Where $K$ is a constant that depends on $\beta,\alpha_{\mathcal{Z}},\alpha_{\tilde{c}}$ and $\|\tilde{c}\|_\infty$. (The exact expression for the constant can be found in the proof).
\end{theorem}

\begin{remark}
We note that the upper bound for $\beta$ can be chosen as $\frac{1}{(2+\|L\|_\infty)\alpha_{\mathcal{Z}}+1}$,  instead of $4\alpha_\mathcal{Z}+1$ (see Remark \ref{beta_remark} in Appendix), where 
\begin{align*}
\|L\|_\infty=\sup_{\pi\in\P(\mathds{X})}\sup_{\gamma\in\Gamma}\sup_{Y_{[0,N]},U_{[0,N-1]}} \rho_{BL}\left(P^{\pi}(\cdot|Y_{[0,N]},U_{[0,N-1]}),P^{\hat{\pi}}(\cdot|Y_{[0,N]},U_{[0,N-1]})\right).
\end{align*}
Hence, under a uniform filter stability bound, the upper bound can be chosen near $\frac{1}{2\alpha_{\mathcal{Z}}+1}$.
\end{remark}
The proof of this theorem is rather long, and accordingly, it is presented in Section \ref{proofThmMain}.

Before presenting the main result of the paper, we provide further supporting results that will let us work with a probability density function. 
\begin{lemma}\label{density_tv}
Assume that the transition kernel $\mathcal{T}(dx_1|x_0,u_0)$ admits a density function $f$ with respect to a reference measure $\phi$ such that $\mathcal{T}(dx_1|x_0,u_0)=f(x_1,x_0,u_0)\phi(dx_1)$.
If $|f(x_1,x_0,u_0)-f(x_1,{x_0}',u_0)|\leq \alpha_{\mathds{X}} |x_0-{x_0}'|$ for all $x_1,x_0,{x_0}'\in\mathds{X}$ and $u_0\in\mathds{U}$ then
\[\|\mathcal{T}(\cdot|x_0,u_0)-\mathcal{T}(\cdot|{x_0}',u_0)\|_{TV}\leq \alpha_{\mathds{X}}|x-x'|.\]
\end{lemma}

We note that using this result, assumptions of Theorem \ref{belief_kernel_regularity} can be expressed with the Lipschitz condition on the density function noted above. We now restate the assumptions that will be used for the main result.
\begin{assumption}\label{main}
\hfill
\begin{itemize}
\item[1.]  The transition kernel $\mathcal{T}(dx_1|x_0,u_0)$ admits a density function $f$ with respect to a reference measure $\phi$ such that $\mathcal{T}(dx_1|x_0,u_0)=f(x_1,x_0,u_0)\phi(dx_1).$
\item[2.] There exists some $\alpha_{\mathds{X}}<\infty$ such that $|f(x_1,x_0,u_0)-f(x_1,{x_0}',u_0)|\leq \alpha_{\mathds{X}} |x_0-{x_0}'|.$
\item[3.] There exists some $\alpha_c<\infty$ such that for all $u\in\mathds{U}$, $|c(x,u)-c(x',u)|\leq \alpha_c|x-x'|.$

\item[4.] $\alpha:=(1-\tilde{\delta}(\mathcal{T}))(2-\delta(Q))<1$, 
\item[5.] The transition kernel $\mathcal{T}$ is dominated, i.e. there exists a dominating measure $\hat{\pi} \in \P(\mathds{X})$ such that for every $x\in\mathds{X}$ and $u \in \mathds{U}$, $\mathcal{T}(\cdot|x,u)\ll\hat{\pi}(\cdot)$ that is $\mathcal{T}(\cdot|x,u)$ is absolutely continuous with respect to $\hat{\pi}$ for every $x\in\mathds{X}$ and $u \in \mathds{U}$. 
\end{itemize}
\end{assumption}

Before the result, we discuss the Lipschitz constants of interest and their relation to each other. First, note that by Lemma \ref{density_tv}, Assumption \ref{main} (second item) implies that  $\|\mathcal{T}(\cdot|x_0,u_0)-\mathcal{T}(\cdot|{x_0}',u_0)\|_{TV}\leq \alpha_{\mathds{X}}|x-x'|$. Hence, by Theorem \ref{belief_kernel_regularity}, we can have various bounds on $\alpha_\mathcal{Z}$, where $\rho_{BL}(\eta(\cdot|z,u),\eta(\cdot|z',u))\leq \alpha_{\mathcal{Z}}\rho_{BL}(z,z')$. In particular, we have  $\alpha_{\mathcal{Z}}\leq (3-2\delta(Q))(1+\alpha_{\mathds{X}})$. 

\begin{theorem}\label{mainThmNIPS}
Assume that we let the system start running for $N$ time steps under a known policy $\gamma$ (which may not be optimal), and the finite window policy starts acting after observing $N$ information variables at time $t=N$. Under Assumption \ref{main}, if $\beta<\frac{1}{4\alpha_\mathcal{Z}+1}$ we have
\begin{align*}
&E_\mu^{\gamma}\left[J_\beta(\pi_{N_-},\mathcal{T},\gamma^*_N)-J_\beta(\pi_{N_-},\mathcal{T},\gamma^*)|Y_{[0,N]},U_{[0,N-1]}\right]\leq K(\beta,\alpha_{\mathds{X}},\alpha_c,\|c\|_\infty) \alpha^N.
\end{align*}
where $\gamma_N^*$ is the optimal finite window policy and where  $K(\beta,\alpha_{\mathds{X}},\alpha_{c},\|c\|_\infty)$ is a constant that depends on $\beta,\alpha_{\mathds{X}},\alpha_{c}$ and $\|c\|_\infty$ (The exact formula for the constant can be found in the appendix in (\ref{KDefinition1})). 

In the statement, $J_\beta(\pi_{N_-},\mathcal{T},\gamma^*_N)$ (respectively $J_\beta(\pi_{N_-},\mathcal{T},\gamma^*)$) denotes the cost under the policy $\gamma^*_N$ (respectively $\gamma^*$) when the initial prior distribution is $\pi_{N_-}$, where $\pi_{N_-}=Pr(X_N\in\cdot|Y_{[0,N-1]},U_{[0,N-1]})$ and the expectation is with respect to the realizations of $Y_{[0,N-1]},U_{[0,N-1]}$.
\end{theorem}


\begin{remark} We note that Theorem \ref{mainThmNIPS} applies to all finite state, measurement, action models as long as $\alpha=(1-\tilde{\delta}(\mathcal{T}))(2-\delta(Q))<1$ and $\beta$ satisfies the condition noted. Example \ref{doub_exmp} shows how to calculate the Dobrushin coefficient for transition matrices in finite setups. All other conditions apply since all probability measures on a finite/countable set are majorized by a probability measure which places a positive mass on every single point. Notice that the third condition only requires that the difference in the cost is bounded for every fixed control action. Condition 1 and 5 of Assumption \ref{main} coincide in the finite case. 
\end{remark}

\begin{remark}
We note that the error bound of the result is independent of the chosen $\hat{\pi}$. As we will see in the following proof, the first upper bound is a result of Theorem \ref{finite_MDP} which indeed depends on the $\hat{\pi}$ chosen by the user. However, thanks to Theorem \ref{curtis_result}, we can get a further upper bound on the error which is uniform over any $\hat{\pi}$ as long as $\hat{\pi}$ is a dominating measure.

One can interpret the absolute continuity assumption, that is $\pi\ll\hat{\pi}$, as follows: assume that the starting distribution of the process is $\pi$ but we start the update from the fixed prior $\hat{\pi}$. The information, $y_{[0,t]},u_{[0,t-1]}$ can eventually fix the starting error uniformly for all such $\hat{\pi}$, as long as, the fixed starting distribution $\hat{\pi}$, puts on a positive measure to every event that the real starting distribution $\pi$ puts on a positive measure. However, if it is not the case, that is, if the incorrect starting distribution $\hat{\pi}$ puts $0$ measure to some event, that $\pi$ puts positive measure to, information variables are not sufficient to fix the starting error occurring from that $0$ measure event. Of course this would not be feasible as the prior would not be compatible with the measured data. In any case, in our setup, the fixed prior $\hat{\pi}$ serves as an approximation and this can be made to satisfy the absolute continuity condition by design.
\end{remark}



\begin{proof}
When we reduce a partially observed MDP to a fully observed process, the initial state of the belief process becomes the Bayesian update of the prior distribution of the state process of the POMDP. Hence, we can write that 
\begin{align*}
&E_\mu^{\gamma}\left[J_\beta(\pi_{N_-},\mathcal{T},\gamma^*_N)-J_\beta(\pi_{N_-},\mathcal{T},\gamma^*)|Y_{[0,N]},U_{[0,N-1]}\right]\\
&=E_\mu^{\gamma}\left[J_\beta(\pi_N,\eta,\gamma_N^*)-J_\beta(\pi_N,\eta,\gamma^*)|Y_{[0,N]},U_{[0,N-1]}\right].
\end{align*}
Notice that, using Theorem \ref{belief_kernel_regularity}, Condition 2 of Assumption \ref{main} implies that
\begin{align*}
\rho_{BL}\left(\eta(\cdot|z,u),\eta(\cdot|z',u)\right)\leq 3(1+\alpha_{\mathds{X}})\rho_{BL}(z,z')
\end{align*}
and we also have that
\begin{align*}
 |\tilde{c}(z,u)-\tilde{c}(z',u)|=\left|\int c(x,u)z(dx)-\int c(x,u)z'(dx)\right|\leq (\alpha_c+\|c\|_\infty)\rho_{BL}(z,z').
\end{align*}
Thus, we can use Theorems \ref{belief_kernel_regularity} and \ref{finite_MDP} to write
\begin{eqnarray} \label{filterSEqnSum}
&&E_\mu^{\gamma}\left[J_\beta(\pi_N,\eta,\gamma_N^*)-J_\beta(\pi_N,\eta,\gamma^*)|Y_{[0,N]},U_{[0,N-1]}\right] \nonumber \\
&&\leq  K(\beta,\alpha_{\mathds{X}},\alpha_{c},\|c\|_\infty) \sup_{\pi\in\P(\mathds{X})}\sup_{\gamma\in\Gamma}E_\pi^{\gamma}\left[\rho_{BL}\left(P^{\pi}(X_N\in\cdot|Y_{[0,N]}),P^{\hat{\pi}}(X_N\in\cdot|Y_{[0,N]})\right)\right]. 
\end{eqnarray}

We set $\hat{\pi}$ in Assumption \ref{main} as our representative probability measure for the quantization of the belief space. Notice that by our choice of $\hat{\pi}$ is a dominating measure and therefore $\pi_{t-}\ll \hat{\pi}$ for any time step $t$, where $\pi_{t-}$ is the predictor at time $t$ . Thus, Theorem \ref{curtis_result} yields that under Assumption \ref{main} 
\begin{align*}
 E_\pi^{\gamma}\left[\|P^{\pi}(\cdot|Y_{[0,N]})-P^{\hat{\pi}}(\cdot|Y_{[0,N]})\|_{TV}\right]\leq \alpha^N,
\end{align*}
for any predictor $\pi$.
We also have by definition that
\begin{align*}
\rho_{BL}\left(P^{\pi,\gamma}(\cdot|Y_{[0,N]}),P^{\hat{\pi},\gamma}(\cdot|Y_{[0,N]})\right)\leq \|P^{\pi,\gamma}(\cdot|Y_{[0,N]})-P^{\hat{\pi},\gamma}(\cdot|Y_{[0,N]})\|_{TV}.
\end{align*}
Hence the result follows.
\end{proof}

\begin{remark}
By studying (\ref{filterSEqnSum}) and utilizing Theorem \ref{belief_kernel_regularity}(i)-(iii), we can arrive at complementary results when we only have filter stability but not exponential filter stability.
\end{remark}

We note that the initial $N$ time steps of the control problem should be treated separately as our approximation technique uses information variables with size $N$, but for the initial $N$ steps, there is not enough observation or control action variables to make use of the approximate finite window policies. In the following result, for the first $N$ time steps, we use a control policy that is found with a policy iteration type argument where the terminal cost estimated with $ \beta^NE[J_\beta(\gamma^*_N)]$ with $\gamma_N^*$ being the approximate finite window policy.

The following result is stated for the best possible policy in the set $\Gamma^N$ (see Equation \ref{new_info}), since $\inf_{\gamma^N\in\Gamma^N}J_\beta(\mu,\mathcal{T},\gamma^N)=J_\beta^N(\mu,\mathcal{T})$, however, in the proof, instead of working with the policy achieving this infimum, we construct a policy (possibly sub-optimal) which achieves the stated upper bound, and this upper bound naturally becomes an upper bound for the best possible finite window policy in $\Gamma^N$.

\begin{theorem}\label{mainThmNIPS2}
Under Assumption \ref{main}, if $\beta<\frac{1}{4\alpha_\mathcal{Z}+1}$ we have
\begin{align*}
&J_\beta^N(\mu,\mathcal{T})-J_\beta^*(\mu,\mathcal{T})\leq  K(\beta,\alpha_{\mathds{X}},\alpha_c,\|c\|_\infty)\alpha^N\beta^N
\end{align*}
where $K(\beta,\alpha_{\mathds{X}},\alpha_{c},\|c\|_\infty)$ is a constant that depends on $\beta,\alpha_{\mathds{X}},\alpha_{c}$ and $\|c\|_\infty$. (The exact formula for the constant can be found in the appendix). 
\end{theorem}

\begin{proof}
Recall the optimal policy for the finite model constructed in Section \ref{finite_state} is denoted by $\gamma_N^*$. Notice that $\gamma_N^*$ is a stationary policy and optimal for any initial point since it solves the discounted cost optimality equation. 

Now, we construct the following policy. Use the policy $\gamma_N^*$, after time $N$ unaltered, but modify the first $N$ time-stage policies, as a batch update, which can be solved via a finite dynamic programming algorithm.

\[\tilde{\gamma}_0,…,\tilde{\gamma}_{N-1} = \mathrm{argmin}_{\gamma_1,…,\gamma_{N-1}} E[\sum_{k=0}^{N-1} \beta^k c(x_k,u_k)] + \beta^k E[J_\beta(\gamma^*_N) | I_N]\]
where $I_N$ is the history by time $N$. We denote this policy by $\gamma^{++}:= \{\tilde{\gamma}_0,…,\tilde{\gamma}_{N-1}, \gamma^{*}_{N},\gamma^*_N,\dots\}$.

Note that we denote the true optimal policy for the original model by $\gamma^*$. We now define the policy $\gamma^+:= \{\gamma^*_{[0,N-1]},  \gamma^{*}_{N},\gamma^*_N,\dots \}$: Apply $\gamma^*$ until time $N$, and then use our finite window policy $\gamma^{*N}$. Note that this policy is not practical as it assumes that the controller already knows the true optimal policy $\gamma^*$, however, we will make use this hypothetical policy for the analysis.

From the way $\gamma^+$ is constructed, we clearly, $J_\beta(\gamma^{++}) \leq J_\beta(\gamma^+)$. And thus,

\[J_\beta(\gamma^{++}) - J_\beta(\gamma^*) \leq J_\beta(\gamma^+) - J_\beta(\gamma^*)\]

Furthermore, because of the way we constructed them, we have $\gamma^{++}\in\Gamma^N$. Hence, we can write 
\begin{align*}
J_\beta^N(\mu,\mathcal{T})-J_\beta^*(\mu,\mathcal{T})&=\inf_{\gamma\in\Gamma^N}J_\beta(\mu\,\mathcal{T},\gamma)-J_\beta(\mu,\mathcal{T},\gamma^*)\\
&\leq J_\beta(\mu,\mathcal{T},\gamma^{++})-J_\beta(\mu,\mathcal{T},\gamma^*)\\
&\leq J_\beta(\mu,\mathcal{T},\gamma^{+})-J_\beta(\mu,\mathcal{T},\gamma^*)
\end{align*}
Then, we start analyzing the error term: because $\gamma^+$ and $\gamma^*$ use the same policy before time $N$, we have
\begin{align}\label{mid_step}
&J_\beta(\mu,\mathcal{T},\gamma^+)-J_\beta(\mu,\mathcal{T},\gamma^*)\nonumber\\
&=\sum_{t=N}^\infty\beta^t E_\mu\left[c(X_t,\gamma_N^*(Y_{[t-N,t]},U_{[t-N,t-1]})\right]-\sum_{t=N}^\infty\beta^t E_\mu\left[c(X_t,\gamma^*(Y_{[0,t]},U_{[0,t-1]})\right]\nonumber\\
&=\beta^N\sum_{t=N}^\infty\beta^{t-N}E_\mu^{\gamma^*}\bigg[ E_\mu\left[c(X_t,\gamma_N^*(Y_{[t-N,t]},U_{[t-N,t-1]})\right]\nonumber\\
&\qquad\qquad\qquad\qquad\qquad- E_\mu\left[c(X_t,\gamma^*(Y_{[0,t]},U_{[0,t-1]})\right]|Y_{[0,N]},U_{[0,N-1]}\bigg]\nonumber\\
&=\beta^NE_\mu^{\gamma^*}\left[J_\beta(\pi_{N_-},\mathcal{T},\gamma^*_N)-J_\beta(\pi_{N_-},\mathcal{T},\gamma^*)|Y_{[0,N]},U_{[0,N-1]}\right]
\end{align}
The last step follows from the observation that conditioning on the observations and control actions $Y_{[0,N]},U_{[0,N-1]}$, the state process can be thought as if it starts at time $t=N$ whose prior measure is $\pi_{N_-}$ and from the fact that the probability measure of $Y_{[0,N]},U_{[0,N-1]}$ is determined by the initial measure $\mu$ and the policy $\gamma^*$ since $\gamma_t^N=\gamma_t^*$ for $t\leq N$. 

The result then follows from Theorem \ref{mainThmNIPS}.
\end{proof}

\subsection{Analysis via Uniform Bounds on Filter Approximation Error}\label{unif_quant}
The following result provides an alternative setup where our analysis is applicable when the filter stability error is presented in terms of a uniform bound on the filter approximation error under the total variation distance. We define this bound, arising from filter stability error, as follows:
\begin{align}\label{TVUnifB}
\bar{L}_{TV}:=\sup_{\pi\in\P(\mathds{X})}\sup_{\gamma\in\Gamma}\sup_{y_{[0,N]},u_{[0,N-1]}} \left\|P^{\pi}(\cdot|y_{[0,N]},u_{[0,N-1]})-P^{\hat{\pi}}(\cdot|y_{[0,N]},u_{[0,N-1]})\right\|_{TV}.
\end{align}




\begin{theorem}\label{unif_cont}
If $\beta <\frac{1}{\alpha_{\mathcal{Z}}}$, where $\alpha_{\mathcal{Z}}$ can be chosen as $(3-2\delta(Q))(1-\tilde{\delta}(\mathcal{T}))$ by Theorem \ref{belief_kernel_regularity}, we have that
\begin{itemize}
\item[(i)]\begin{align*}
\sup_z\left|J_\beta^N(z) - J^*_\beta(z) \right|\leq\frac{(\alpha_{\mathcal{Z}}-1)\beta+1}{(1-\beta)^2(1-\alpha_{\mathcal{Z}}\beta)}\|c\|_\infty\bar{L}_{TV}.
\end{align*}
\item[(ii)] \begin{align*}
\sup_z\left|J_\beta(z,\gamma_N)-J^*_\beta(z)\right|\leq \frac{2(1+(\alpha_{\mathcal{Z}}-1)\beta)}{(1-\beta)^3(1-\alpha_{\mathcal{Z}}\beta)}\|c\|_\infty \bar{L}_{TV}.
\end{align*}

\end{itemize}
\end{theorem}
Before the proof, we note that the bound applies for all $\beta < 1$, if $(3 - 2 \delta(Q)) (1-\tilde{\delta}(\mathcal{T}))< 1$. Furthermore, since $(3 - 2 \delta(Q)) (1-\tilde{\delta}(\mathcal{T})) = 2 \alpha - (1-\tilde{\delta}(\mathcal{T}))$ where $\alpha=(2-\delta(Q))(1-\tilde{\delta}(\mathcal{T}))$ and, therefore, if $\alpha< 1$ we have that for all $\beta < 1/(1+\tilde{\delta}(\mathcal{T}))$ the result applies independent of $\delta(Q)$, though this is a conservative bound. If $\delta(Q) =1$, the bound applies for all $\beta < 1$ as long as $\alpha < 1$.
\begin{proof}
(i) We start by writing the fixed point equations
\begin{align*}
&J_\beta^*(z)=\min_u\left(\tilde{c}(z,u)+\beta\int J_\beta^*(z_1)\eta(dz_1|z,u)\right)\\
&J^N_\beta(z)=\min_u\left(\tilde{c}(F(z),u)+\beta\int J_\beta^N(z_1)\eta(dz_1|F(z),u)\right).
\end{align*}
Hence we can write that
\begin{align}\label{boundUnif}
\sup_z&\left|J^*_\beta(z)-J_\beta^N(z)\right|\leq \sup_u|\tilde{c}(z,u)-\tilde{c}(F(z),u)| \nonumber \\
&\qquad+\beta\sup_u\left|\int J_\beta^*(z_1)\eta(dz_1|z,u)-\int J_\beta^N(z_1)\eta(dz_1|F(z),u)\right| \nonumber \\
&\leq \|c\|_\infty \|z-F(z)\|_{TV}+\beta \sup_u \int\left|J_\beta^*(z_1)-J_\beta^N(z_1)\right|\eta(dz_1|F(z),u) \nonumber \\
&\qquad+\beta \sup_u\left|\int J_\beta^*(z_1)\eta(dz_1|z,u)-\int J_\beta^*(z_1)\eta(dz_1|F(z),u)\right| \nonumber \\
&\leq \|c\|_\infty  \|z-F(z)\|_{TV} +\beta \sup_z\left|J^*_\beta(z)-J_\beta^N(z)\right|+\beta \|J_\beta^*\|_{BL}\alpha_{\mathcal{Z}} \|z-F(z)\|_{TV}.
\end{align}
Note that by Theorem \ref{belief_kernel_regularity}, $\alpha_{\mathcal{Z}}\leq (3-2\delta(Q))(1-\tilde{\delta}(\mathcal{T}))$ when $\mathcal{Z}$ is associated with total variation distance. We also have that $\|J_\beta^*\|_{BL}\leq\frac{2-\beta}{(1-\beta)(1-\alpha_{\mathcal{Z}}\beta)}\|c\|_\infty$ when $\mathcal{Z}$ metrized by total variation distance (see Lemma \ref{lip_val}). Hence, by noting $ \|z-F(z)\|_{TV}\leq \bar{L}_{TV}$,  we can conclude that
\begin{align*}
\sup_z\left|J^*_\beta(z)-J_\beta^N(z)\right|\leq \frac{(\alpha_{\mathcal{Z}}-1)\beta+1}{(1-\beta)^2(1-\alpha_{\mathcal{Z}}\beta)}\|c\|_\infty\bar{L}_{TV}.
\end{align*}
%

(ii) We start by writing
\begin{align*}
\sup_z\left|J_\beta(z,\gamma_N)-J^*_\beta(z)\right|\leq \sup_z\left|J_\beta(z,\gamma_N)-J^N_\beta(z)\right|+\sup_z\left|J^N_\beta(z)-J^*_\beta(z)\right|
\end{align*}
where the second term is bounded by (i). We now focus on the first term:
\begin{align*}
 \left|J_\beta(z,\gamma_N)-J^N_\beta(z)\right|\leq &\sup_u\left|\tilde{c}(z,u)-\tilde{c}(F(z),u)\right|\\
&+\beta\left|\int J_\beta(z_1,\gamma_N)\eta(dz_1|z,\gamma_N(z))-\int J_\beta^N(z_1)\eta(dz_1|F(z),\gamma_N(z))\right|\\
&\leq \|c\|_\infty \left\|z-F(z)\right\|_{TV}+\beta\int \left|J_\beta(z_1,\gamma_N)-J_\beta^N(z_1)\right|\eta(dz_1|z,\gamma_N(z))\\
&\qquad+\beta\int \left|J_\beta^N(z_1)-J_\beta^*(z_1)\right|\eta(dz_1|z,\gamma_N(z))\\
&\qquad +\beta\left|\int J_\beta^*(z_1)\eta(dz_1|z,\gamma_N(z))-\int J_\beta^*(z_1)\eta(dz_1|F(z),\gamma_N(z))\right|\\
&\qquad+ \beta \int \left|J_\beta^*(z_1)-J_\beta^N(z_1)\right|\eta(dz_1|F(z),\gamma_N(z))\\
&\leq \|c\|_\infty \left\|z-F(z)\right\|_{TV}+ \beta \sup_z  \left|J_\beta(z,\gamma_N)-J^N_\beta(z)\right|\\
&+2\beta\sup_z\left|J^N_\beta(z)-J^*_\beta(z)\right|+\beta\|J_\beta^*\|_{BL}\alpha_{\mathcal{Z}}  \left\|z-F(z)\right\|_{TV}.
\end{align*}
Thus, using (i), and $\|J_\beta^*\|_{BL}\leq\frac{2-\beta}{(1-\beta)(1-\alpha_{\mathcal{Z}}\beta)}\|c\|_\infty$,  we can write
\begin{align*}
\sup_z \left|J_\beta(z,\gamma_N)-J^N_\beta(z)\right|\leq \frac{(1+\beta)((\alpha_{\mathcal{Z}}-1)\beta+1)}{(1-\beta)^3(1-\alpha_{\mathcal{Z}}\beta)}\|c\|_\infty \bar{L}_{TV}.
\end{align*} We then conclude that
\begin{align*}
\sup_z\left|J_\beta(z,\gamma_N)-J^*_\beta(z)\right|&\leq \sup_z\left|J_\beta(z,\gamma_N)-J^N_\beta(z)\right|+\sup_z\left|J^N_\beta(z)-J^*_\beta(z)\right|\\
&\leq\frac{(\alpha_{\mathcal{Z}}-1)\beta+1}{(1-\beta)^2(1-\alpha_{\mathcal{Z}}\beta)}\|c\|_\infty\bar{L}_{TV}.+  \frac{(1+\beta)((\alpha_{\mathcal{Z}}-1)\beta+1)}{(1-\beta)^3(1-\alpha_{\mathcal{Z}}\beta)}\|c\|_\infty \bar{L}_{TV}\\
&=  \frac{2(1+(\alpha_{\mathcal{Z}}-1)\beta)}{(1-\beta)^3(1-\alpha_{\mathcal{Z}}\beta)}\|c\|_\infty \bar{L}_{TV}.
\end{align*}
\end{proof}

\subsection{A Discussion on the Controlled Filter Stability Problem}
The above results suggest that the loss occurring from applying a finite window policy is mainly controlled by the term 
\begin{align}\label{filt_rmrk}
\sup_{\pi\in\P(\mathds{X})}\sup_{\gamma\in\Gamma} E_\pi^{\gamma}\left[\rho_{BL}\left(P^{\pi}(\cdot|Y_{[0,N]},U_{[0,N-1]}),P^{\hat{\pi}}(\cdot|Y_{[0,N]},U_{[0,N-1]})\right)\right],
\end{align} 
or its variations where the $BL$ metric is replaced with total variation, or the expectation is replaced with a uniform bound (given in (\ref{TVUnifB})). All these terms describe how fast two different belief-state processes forget their initial priors when fed with the same observations/control actions under a true distribution. Thus, any bound for this term directly applies to the main results we presented for the loss caused by a finite window policy. This term is related to the filter stability problem and our approximation results point out the close relation between filter stability and the performance of finite window policies. In a way, the main result or message of this paper is perhaps to explicitly relate finite window approximations for POMDPs to the filter stability problem.

To bound the term (\ref{filt_rmrk}), we use Theorem \ref{curtis_result}, to achieve an exponential convergence rate in the window size for a controlled setup. However, we should note that, Theorem \ref{curtis_result} provides only a sufficient condition to bound the filter stability term geometrically fast in the total variation distance. This result can be seen as too strong, if one is only interested in making this $\rho_{BL}$ distance (\ref{filt_rmrk}) smaller with increasing window size. In fact, as we will see in Section \ref{num_study}, even when the assumptions of Theorem \ref{curtis_result} are not satisfied, the filter stability term still converges to $0$. In the literature, there are various set of assumptions to achieve filter stability. Two main approaches have been:
\begin{itemize}
\item The transition kernel is in some sense sufficiently ergodic, forgetting the initial
measure and therefore passing this insensitivity (to incorrect initializations)
on to the filter process. This condition is often tailored towards control-free models.
\item The measurement channel provides sufficient information about the underlying state, allowing the filter to track the true state process. This approach is typically based on martingale methods and accordingly does not often lead to rates of convergence for the filter stability problem, but only asymptotic filter stability.
\end{itemize}
The result we use in this paper (Theorem \ref{curtis_result}) provides exponential filter stability, using a joint contraction property of the Bayesian filter update and measurement update steps through the Dobrushin coefficient. When these requirements are not satisfied, the filter stability can be checked via different assumptions from the literature.  However, we also note that, for the controlled setup, filter stability results are limited compared to the control free setup. A comprehensive review on filter stability in the control-free case, is available by \citet{chigansky2009intrinsic}. In the controlled case, recent additional studies include \cite{MYCDC2019observability,MYRobustControlledFS} where martingale methods are used to arrive at controlled filter stability, which in turn can lead to weaker conditions (though without rates of convergence) for near optimality of finite window policies.


With regard to Theorem \ref{curtis_result}, the following example studies an additive Gaussian model (not necessarily linear) and provides different parameters for the condition $(1-\delta(\mathcal{T}))\times(2-\delta(\hat{Q}))<1$ to hold.
\begin{example}
Consider a system where $\mathds{X}=\mathds{Y}=\mathds{R}$ and the transition and measurement kernels
are given by
\begin{align*}
x_{n+1}=f(x_n,u_n)+N(0,\sigma_t^2),\qquad y_n=g(x_n)+N(0,\sigma_q^2)
\end{align*}
where the functions $f$ and $g$ are measurable and bounded such that $f(x,u)\in[-t,t]$ and $g(x)\in [-q,q]$.

Note that, in our paper we assume and present our results from finite $\mathds{Y}$. Therefore, to make the example compatible with our results, we discretize the observation space $\mathds{Y}$. We provide two discretization schemes one with $\hat{\mathds{Y}}_1=\{-q,q\}$ and the other with $\hat{\mathds{Y}}_2=\{-q,0,q\}$. For discretization, we use a nearest neighbor mapping.

First, we study the observation space $\hat{\mathds{Y}}_1=\{-q,q\}$. Using the nearest neighbor mapping, we have that 
\begin{align*}
&\hat{y}_n=-q\quad  \text{ if } y_n=g(x_n)+N(0,\sigma_q^2)\leq 0,\\
&\hat{y}_n=q\quad  \text{ if } y_n=g(x_n)+N(0,\sigma_q^2)> 0
\end{align*}
We then can write the following for the transition kernel $\mathcal{T}(\cdot|x,u)\in\P(\mathds{X})$
\begin{align*}
\mathcal{T}(dx_{t+n}|x_n,u_n)\sim N(f(x_n,u_n),\sigma_t^2) 
\end{align*}
and for the compound channel $\hat{Q}(\cdot|x)\in\P(\hat{\mathds{Y}}_1)$:
\begin{align*}
\hat{Q}(q|x_n)=Pr(N(g(x_n),\sigma_q^2)>0),\qquad \hat{Q}(-q|x_n)=Pr(N(g(x_n),\sigma_q^2)\leq0).
\end{align*}
For these kernels, the Dobrushin coefficients can be calculated as 
\begin{align*}
\delta(\mathcal{T})\geq2Pr(N(t,\sigma_t^2)\leq 0),\qquad \delta(\hat{Q})=2Pr(N(q,\sigma_q^2)\leq 0).
\end{align*}
Notice that these probabilities are fully determined by the ratio of the mean and standard deviation of the Gaussian in question, $\sigma_t/t$ and $\sigma_q/q$ . The higher the ratio, the higher the Dobrushin coefficient. Below, we see a list of the ratio of the transition kernel and lowest possible
ratio of the measurement kernel such that $(1-\delta(\mathcal{T}))\times(2-\delta(\hat{Q}))<1$. If the ratio of $\sigma_q/q$ is
higher than the stated value, we will get exponential stability for the given transition
kernel. Note that if $\sigma_t/t>1.5$, then $\delta(\mathcal{T})>0.5$ which makes $(1-\delta(\mathcal{T}))\times(2-\delta(\hat{Q}))<1$ regardless of the channel for which we use 'any' in table to indicate that any channel would lead to exponential filter stability. 

\begin{table}[h!]
\begin{center}
\begin{tabular}{ |c|c|c|c|c|c|c|c|c|c|c|c|c|c| } 
 \hline
 $\frac{\sigma_t}{t}$ & 1.5 & 1.4 & 1.3 & 1.2 & 1.1 & 1.0 & 0.9 & 0.8 & 0.7 & 0.6 & 0.5 & 0.4 & 0.3 \\ 
\hline
$\frac{\sigma_q}{q}$ & any & 0.6 & 0.8 & 1.01 & 1.3 & 1.65 & 2.13 & 3.25 & 5.5 & 8.0 & 20.0 & 70.0 & 1000.0 \\ 
\hline
$\delta(\mathcal{T})$ & 0.50 & 0.48  & 0.44 & 0.40 & 0.36 & 0.32 & 0.27 & 0.21 & 0.15 & 0.10 & 0.05 & 0.01& 0.00 \\ 
 \hline
$\delta(\hat{Q})$ & any & 0.1 & 0.21 & 0.32 & 0.44 & 0.54 & 0.64 & 0.76 & 0.86 & 0.90 & 0.96 & 0.99 & 1.00\\ 
 \hline
\end{tabular}
\end{center}
\caption{Approximate minimum ratio of $\frac{\sigma_q}{q}$ for exponential filter stability for $\hat{\mathds{Y}}_1=\{-q,q\}$}
\label{table:1}
\end{table}

We now analyze the problem for the observation space $\hat{\mathds{Y}}_2=\{-q,0,q\}$. For this set, the nearest neighbor mapping yields that
\begin{align*}
&\hat{y}_n=-q\quad  \text{ if } y_n=g(x_n)+N(0,\sigma_q^2)\leq \frac{-q}{2},\\
&\hat{y}_n=q\quad  \text{ if } y_n=g(x_n)+N(0,\sigma_q^2)> \frac{q}{2},\\
&\hat{y}_n=0\quad  \text{ else} .
\end{align*}
For the compound channel, we then have
\begin{align*}
&\hat{Q}(q|x_n)=Pr(N(g(x_n),\sigma_q^2)>\frac{q}{2}),\qquad \hat{Q}(-q|x_n)=Pr(N(g(x_n),\sigma_q^2)\leq\frac{-q}{2}),\\
&\hat{Q}(0|x_n)=Pr(\frac{q}{2}\geq N(g(x_n),\sigma_q^2)>\frac{-q}{2}).
\end{align*}
For these kernels, the Dobrushin coefficients can be calculated as
\begin{align*}
&\delta(\mathcal{T})=2Pr(N(t,\sigma_t^2)\leq 0),\\ 
&\delta(\hat{Q})=2Pr(N(q,\sigma_q^2)\leq \frac{-q}{2})+Pr(\frac{-q}{2}<N(q,\sigma_q^2)<\frac{q}{2}).
\end{align*}
Below, we now see a list of the ratio of the transition kernel and lowest possible
ratio of the measurement kernel such that $(1-\delta(\mathcal{T}))\times(2-\delta(\hat{Q}))<1$ for the observation space $\hat{\mathds{Y}}_2=\{-q,0,q\}$. 

\begin{table}[h!]
\begin{center}
\begin{tabular}{ |c|c|c|c|c|c|c|c|c|c|c|c|c|c| } 
 \hline
 $\frac{\sigma_t}{t}$ & 1.5 & 1.4 & 1.3 & 1.2 & 1.1 & 1.0 & 0.9 & 0.8 & 0.7 & 0.6 & 0.5 & 0.4 & 0.3 \\ 
\hline
$\frac{\sigma_q}{q}$ & any & 0.39 & 0.6 & 0.85 & 1.2 & 1.54 & 2.1 & 3.2 & 5.9 & 8.0 & 20.0 & 80.0 & 1000.0 \\ 
\hline
$\delta(\mathcal{T})$ & 0.50 & 0.48  & 0.44 & 0.40 & 0.36 & 0.32 & 0.27 & 0.21 & 0.15 & 0.10 & 0.05 & 0.01& 0.00 \\ 
 \hline
$\delta(\hat{Q})$ & any & 0.1 & 0.21 & 0.32 & 0.44 & 0.54 & 0.64 & 0.76 & 0.86 & 0.90 & 0.96 & 0.99 & 1.00\\ 
 \hline
\end{tabular}
\end{center}
\caption{Approximate minimum ratio of $\frac{\sigma_q}{q}$ for exponential filter stability for $\hat{\mathds{Y}}_2=\{-q,0,q\}$}
\label{table:2}
\end{table}
\end{example}


\section{Numerical Study}\label{num_study}

In this section, we give an outline of the algorithm used to determine the approximate belief MDP and the finite window policy. We also present the performance of the finite window policies and the value function error for the approximate belief MDP in relation to the window size.

A summary of the algorithm is as follows:
\begin{itemize}
\item We determine a $\hat{\pi}\in\P(\mathds{X})$ such that it puts positive measure over the state space $\mathds{X}$. 
\item Following Section \ref{finite_state}, we construct the finite belief space $\mathcal{Z}_{\hat{\pi}}^N$ by taking the Bayesian update of $\hat{\pi}$ for all possible realizations of the form $\{y_0,\dots,y_N,u_0,\dots,u_{N-1}\}$. Hence we get a approximate finite belief space with size $|\mathds{Y}|^{N+1}\times|\mathds{U}|^N$.
\item We calculate all transitions from each $z\in \mathcal{Z}_{\hat{\pi}}^N$ by considering every possible observation $y$ and control action $u$ and we map them to $\mathcal{Z}_{\hat{\pi}}^N$ using a nearest neighbor map and construct the transition kernels $\eta^N$ for the finite model.
\item For the finite models obtained, through value or policy iteration, we calculate the value functions and optimal policies. 
\end{itemize}

The example we use is a machine repair problem. In this model, we have $\mathds{X,Y,U}=\{0,1\}$ with
\begin{align*}
x_t=&\begin{cases}1 \quad \text{ machine is working at time t }\\
0  \quad \text{ machine is not working at time t }.\end{cases}
u_t=&\begin{cases}1 \quad \text{ machine is being repaired at time t }\\
0  \quad \text{ machine is not being repaired at time t }.\end{cases}
\end{align*}
The probability that the repair was successful given initially the machine was not working is given by $\kappa$:
\begin{align*}
Pr(x_{t+1}=1|x_t=0,u_t=1)=\kappa
\end{align*}
The probability that the machine breaks down while in a working state is given by $\theta$:
\begin{align*}
Pr(x_t=0|x_t=1,u_t=0)=\theta
\end{align*}
The probability that the channel gives an incorrect measurement is given by $\epsilon$:
\begin{align*}
Pr(y_t=1|x_t=0)=Pr(y_t=0|x_t=1)=\epsilon
\end{align*}
The one stage cost function is given by
\begin{align*}
c(x,u)=&\begin{cases}R+E \quad  &x=0,u=1 \\
E  \quad  &x=0, u=0 \\
0 \quad &x=1,u=0\\
R \quad &x=1, u=1\end{cases}
\end{align*}
where $R$ is the cost of repair and 
$E$ is the cost incurred by a broken machine.

We study the example with discount factor $\beta=0.8$,  and $R=5, E=1$ and present three different results by changing the other parameters. For all different cases, we choose $\hat{\pi}(\cdot)=0.1\delta_{0}(\cdot)+0.9\delta_{1}(\cdot)$.

For the first case, we take $\epsilon=0.3$, $\kappa=0.2$, $\theta=0.1$. We analyze the performance for $N\in\{0,\dots,5\}$. To get a proper finite window policy for all $N$'s we use, we let the system run $5$ time steps under the policy $\gamma_0(I_t)=0$ that is $u_t=0$ no matter what for the first $5$ time steps. Then, we start applying the policy and start calculating the cost. In Figure \ref{figure_1}, we plot the approximate value functions and the cost incurred by the finite window policies, that is we plot the terms $E^{\gamma_0}[|J_\beta^N(Z_5)|Y_{[0,5]},U_{[0,4]}]$ and $E^{\gamma_0}[J_\beta(Z_5,\gamma^N)|Y_{[0,5]},U_{[0,4]}]$.
\begin{figure}[h]
\centering
\includegraphics[scale=0.5]{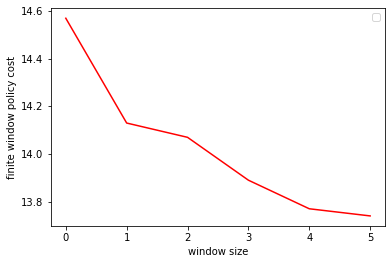}\includegraphics[scale=0.5]{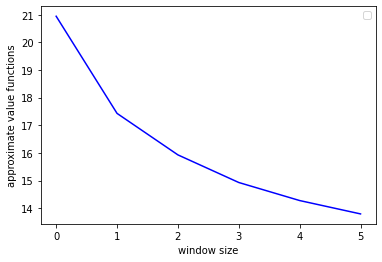}
\caption{Approximate value function and the performance of finite window policy for different window sizes}
\label{figure_1}
\end{figure}


Recall that our main result Theorem \ref{mainThmNIPS}, emphasizes the relation between the filter stability term and 'value error' and 'robustness error', both of which are guaranteed to converge to zero with increasing window length, where
\begin{align*}
&\text{Filter stability term: }\sup_\pi\sup_\gamma E_\pi^{\gamma}\left[\rho_{BL}\left(P^{\pi}(\cdot|Y_{[0,N]}),P^{\hat{\pi}}(\cdot|Y_{[0,N]})\right)\right]\\
&\text{Value error: }E^{\gamma_0}[|J_\beta^N(Z_5)-J_\beta^*(Z_5)|Y_{[0,5]},U_{[0,4]}]\\
&\text{Robustness error: }E^{\gamma_0}[|J_\beta(Z_5,\gamma^N)-J_\beta^*(Z_5)|Y_{[0,5]},U_{[0,4]}]
\end{align*} 
To get the errors, we simply subtract the cost values from their minimum (largest window) which serves as an approximation of the value function. Furthermore, we scale the errors according to the filter stability term to get a better understanding of the error rate in relation to the filter stability constant that is we make them start from the same value to see the decrease rates more clearly. Figure \ref{figure_2} shows the relation.
\begin{figure}[h]
\centering
\includegraphics[scale=0.6]{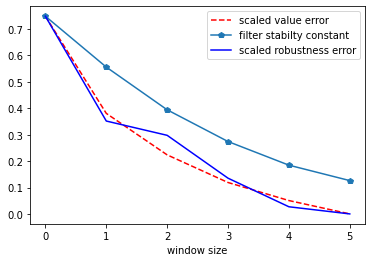}
\caption{Comparison of approximation error and filter stability term for different window sizes}
\label{figure_2}
\end{figure}

It can be seen from Figure \ref{figure_2}, that the error rate stays below the filter stability convergence rate and goes to $0$ as the window size increases.

For the second case, we take  $\epsilon=0.01$, $\kappa=0.3$, $\theta=0.1$. For this case, we directly plot the scaled errors in relation to the filter stability term in Figure \ref{figure_3}.
\begin{figure}[h]
\centering
\includegraphics[scale=0.6]{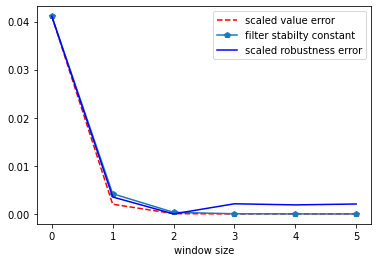}
\caption{Comparison of approximation error and filter stability term for different window sizes for an informative channel}
\label{figure_3}
\end{figure}

It can be seen that as the channel is more informative, the filter stability term gets smaller much faster and the recent information becomes more informative. As a result, we get a faster decrease in the error with the increasing window size.

One can notice that for the first two cases, the Dobrushin coefficient $\alpha$, we defined in Assumption \ref{main} is greater than $1$, however, the error terms still converge to $0$ with increasing window. This is because the filter stability term gets smaller even though the Dobrushin constant $\alpha>1$. For the last case, we select the parameters in a way to make $\alpha<1$.

We take  $\epsilon=0.3$, $\kappa=0.4$, $\theta=0.3$ so that $\alpha=0.7$. 
\begin{figure}[h]
\centering
\includegraphics[scale=0.6]{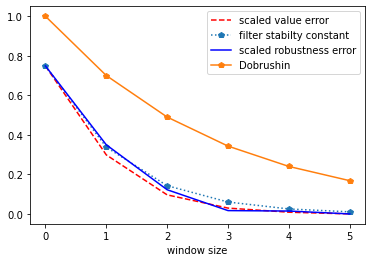}
\caption{Comparison of approximation error, filter stability term and Dobrushin term for different window sizes}
\label{figure_4}
\end{figure}
In Figure \ref{figure_4}, it can be observed that the error terms and the filter stability term converges to $0$ at similar rates. $\alpha^N$, however, gets smaller at a slower rate and upper bounds the convergence rate of error terms.

\section{Proofs of Main Technical Results}
In this section, we provide the technical proofs.

\subsection{Proof of Theorem \ref{belief_kernel_regularity} }\label{SecProofReg}


We will build on the proof by \citet{KSYWeakFellerSysCont}.
 Let $\mathds{X}$ be a separable metric space.
Another metric that metrizes the weak topology on $\P(\mathds{X})$ is the following:
\begin{align}\label{new_weak_metric}
\rho(\mu,\nu) = \sum_{m=1}^{\infty} 2^{-(m+1)} \biggr| \int_{\mathds{X}} f_m(x) \mu(dx) - \int_{\mathds{X}} f_m(x) \nu(dx) \biggl|, 
\end{align}
where $\{f_m\}_{m\geq1}$ is an appropriate sequence of continuous and bounded functions such that $\|f_m\|_{\infty} \leq 1$ for all $m\geq1$ \citep[see][Theorem 6.6, p. 47]{Par67}.

We equip ${\cal P}(\mathds{X})$ with the metric $\rho$ to define bounded-Lipschitz norm $\|f\|_{BL}$ of any Borel measurable function $f:{\cal P}(\mathds{X})\rightarrow \mathds{R}$. With this metric, we can start the proof:
\begin{align}
&\sup_{\|f\|_{BL}\leq1}\bigg|\int_{{\cal P}(\mathds{X})} f(z_1)\eta(dz_1|z_0',u)- \int_{{\cal P}(\mathds{X})} f(z_1)\eta(dz_1|z_0,u)\bigg|\nonumber\\
&=\sup_{\|f\|_{BL}\leq1}\bigg|\int_{\mathds{Y}} f(z_1(z_0',u,y_1))P(dy_1|z_0',u) - \int_{\mathds{Y}} f(z_1(z_0,u,y_1))P(dy_1|z_0,u)\bigg|\nonumber\\
&\leq\sup_{\|f\|_{BL}\leq1} \bigg| \int_{\mathds{Y}} f(z_1(z_0',u,y_1))P(dy_1|z_0',u) - \int_{\mathds{Y}} f(z_1(z_0',u,y_1))P(dy_1|z_0,u) \bigg|\nonumber\\
&\qquad+ \sup_{\|f\|_{BL}\leq1} \int_{\mathds{Y}} \big| f(z_1(z_0',u,y_1)) -f(z_1(z_0,u,y_1)) \big| P(dy_1|z_0,u)\nonumber\\
&\leq \|P(\cdot|z_0',u) - P(\cdot|z_0,u)\|_{TV}\label{first_term}\\
&\qquad+ \sup_{\|f\|_{BL}\leq1}\int_{\mathds{Y}} \big| f(z_1(z_0',u,y_1)) -f(z_1(z_0,u,y_1)) \big|P(dy_1|z_0,u)\label{second_term}
\end{align}
where, in the last inequality, we have used $\|f\|_{\infty} \leq \|f\|_{BL} \leq 1$. 

We first anayze the first term:
\begin{align*}
&\|P(\cdot|z_0',u) - P(\cdot|z_0,u)\|_{TV}=\sup_{\|f\|_\infty\leq 1}|\int f(y_1)P(dy_1|z_0',u) - \int f(y_1)P(dy_1|z_0,u)|\\
&=\sup_{\|f\|_\infty\leq 1} \bigg|\int_{\mathds{X}} \int_{\mathds{X}}\int_{\mathds{Y}} f(y_1)Q(dy_1|x_1)\mathcal{T}(dx_1|x_0,u)z_0'(dx_0) \\
&\phantom{xxxxxxxxxxxxx}- \int_{\mathds{X}} \int_{\mathds{X}}\int_{\mathds{Y}}f(y_1) Q(dy_1|x_1)\mathcal{T}(dx_1|x_0,u)z_0(dx_0) \bigg|.
\end{align*}
For all $x_0', x_0$ and for any $\|f\|_\infty\leq 1$, we have
\begin{align*}
&\left|\int_{\mathds{X}}\int_{\mathds{Y}} f(y_1)Q(dy_1|x_1)\mathcal{T}(dx_1|x_0',u)-\int_{\mathds{X}}\int_{\mathds{Y}} f(y_1)Q(dy_1|x_1)\mathcal{T}(dx_1|x_0,u)\right|\\
&\leq \|\mathcal{T}(\cdot|x_0',u)-\mathcal{T}(\cdot|x_0,u)\|_{TV}\leq \alpha_{\mathds{X}}|x_0'-x_0|.
\end{align*}
Then, we can write
\begin{align*}
&\|P(\cdot|z_0',u) - P(\cdot|z_0,u)\|_{TV} \leq (1+\alpha_{\mathds{X}})\rho_{BL}(z_0',z_0).
\end{align*}

Recall the metric introduced in (\ref{new_weak_metric}); using this metric, we now focus on the term (\ref{second_term}):
\begin{align*}
& \sup_{\|f\|_{BL}\leq1}\int_{\mathds{Y}} \big| f(z_1(z_0',u,y_1))-f(z_1(z_0,u,y_1)) \big| P(dy_1|z_0,u)\\
&\leq  \int_{\mathds{Y}} \rho(z_1(z_0',u,y_1),z_1(z_0,u,y_1))P(dy_1|z_0,u) \\
&= \int_{\mathds{Y}} \sum_{m=1}^{\infty} 2^{-m+1} \bigg| \int_{\mathds{X}} f_m(x_1) z_1(z_0',u,y_1)(dx_1) \nonumber\\
&\phantom{xxxxxxxxxxxxxxxxx}- \int_{\mathds{X}} f_m(x_1) z_1(z_0,u,y_1)(dx_1)\bigg|P(dy_1|z_0,u) \\
&=  \sum_{m=1}^{\infty} 2^{-m+1} \int_{\mathds{Y}} \bigg| \int_{\mathds{X}} f_m(x_1) z_1(z_0',u,y_1)(dx_1) \nonumber \\
&\phantom{xxxxxxxxxxxxxxxxx}- \int_{\mathds{X}} f_m(x_1) z_1(z_0,u,y_1)(dx_1)\bigg|P(dy_1|z_0,u),
\end{align*}
where we have used Fubini's theorem with the fact that $\sup_{m} \|f_m\|_{\infty} \leq 1$. For each $m$, let us define
\begin{align}\label{I_+}
&I_+:=\biggl\{y_1 \in \mathds{Y}: \int_{\mathds{X}} f_m(x_1) z_1(z_0',u,y_1)(dx_1) > \int_{\mathds{X}} f_m(x_1) z_1(z_0,u,y_1)(dx_1) \biggr\}\nonumber\\
&I_-:=\biggl\{y_1 \in \mathds{Y}: \int_{\mathds{X}} f_m(x_1) z_1(z_0',u,y_1)(dx_1) \leq \int_{\mathds{X}} f_m(x_1) z_1(z_0,u,y_1)(dx_1) \biggr\}.
\end{align}

Then, we can write
\begin{align*}
& \int_{\mathds{Y}} \bigg| \int_{\mathds{X}} f_m(x_1) z_1(z_0',u,y_1)(dx_1)  -\int_{\mathds{X}} f_m(x_1) z_1(z_0,u,y_1)(dx_1)\bigg|P(dy_1|z_0,u)\nonumber\\
& =\int_{I_+} \int_{\mathds{X}} f_m(x_1) z_1(z_0',u,y_1)(dx_1)P(dy_1|z_0,u) -\int_{I_+} \int_{\mathds{X}} f_m(x_1) z_1(z_0',u,y_1)(dx_1)P(dy_1|z_0',u)\nonumber\\
& +\int_{I_+} \int_{\mathds{X}} f_m(x_1) z_1(z_0',u,y_1)(dx_1)P(dy_1|z_0',u)- \int_{I_+} \int_{\mathds{X}} f_m(x_1) z_1(z_0,u,y_1)(dx_1)P(dy_1|z_0,u)\nonumber\\
&+\int_{I_-} \int_{\mathds{X}} f_m(x_1) z_1(z_0,u,y_1)(dx_1)P(dy_1|z_0,u)- \int_{I_-} \int_{\mathds{X}} f_m(x_1) z_1(z_0',u,y_1)(dx_1)P(dy_1|z_0',u)\nonumber\\
&+\int_{I_-} \int_{\mathds{X}} f_m(x_1) z_1(z_0',u,y_1)(dx_1)P(dy_1|z_0',u)- \int_{I_-} \int_{\mathds{X}} f_m(x_1) z_1(z_0',u,y_1)(dx_1)P(dy_1|z_0,u)
\end{align*}
For the first and the last term, we can write
\begin{align}
& \int_{I_+} \int_{\mathds{X}} f_m(x_1) z_1(z_0',u,y_1)(dx_1)P(dy_1|z_0,u) -\int_{I_+} \int_{\mathds{X}} f_m(x_1) z_1(z_0',u,y_1)(dx_1)P(dy_1|z_0',u)\nonumber\\
&+\int_{I_-} \int_{\mathds{X}} f_m(x_1) z_1(z_0',u,y_1)(dx_1)P(dy_1|z_0',u)- \int_{I_-} \int_{\mathds{X}} f_m(x_1) z_1(z_0',u,y_1)(dx_1)P(dy_1|z_0,u)\nonumber\\
&=\int_{\mathds{Y}}f_m(x_1)z_1(z_0',u,y_1)(dx_1)\left(\mathds{1}_{I_-}-\mathds{1}_{I_+}\right)P(dy_1|z_0',u)\nonumber\\
&\qquad-\int_{\mathds{Y}}f_m(x_1)z_1(z_0',u,y_1)(dx_1)\left(\mathds{1}_{I_-}-\mathds{1}_{I_+}\right)P(dy_1|z_0,u)\nonumber\\
&\leq \|P(\cdot|z_0',u)-P(\cdot|z_0,u)\|_{TV}\leq (1+\alpha_{\mathds{X}})\rho_{BL}(z_0',z_0).\label{bnd_1}
\end{align}

For the second and the third term:
\begin{align}
& \int_{I_+} \int_{\mathds{X}} f_m(x_1) z_1(z_0',u,y_1)(dx_1)P(dy_1|z_0',u)- \int_{I_+} \int_{\mathds{X}} f_m(x_1) z_1(z_0,u,y_1)(dx_1)P(dy_1|z_0,u)\nonumber\\
&+\int_{I_-} \int_{\mathds{X}} f_m(x_1) z_1(z_0,u,y_1)(dx_1)P(dy_1|z_0,u)- \int_{I_-} \int_{\mathds{X}} f_m(x_1) z_1(z_0',u,y_1)(dx_1)P(dy_1|z_0',u)\nonumber\\
&=\int_{I_+} \int_{\mathds{X}} f_m(x_1) Q(dy_1|x_1)\mathcal{T}(dx_1|z_0',u)-\int_{I_+} \int_{\mathds{X}} f_m(x_1) Q(dy_1|x_1)\mathcal{T}(dx_1|z_0,u)\nonumber\\
&\qquad+\int_{I_-} \int_{\mathds{X}} f_m(x_1)Q(dy_1|x_1)\mathcal{T}(dx_1|z_0,u)-\int_{I_-} \int_{\mathds{X}} f_m(x_1)Q(dy_1|x_1)\mathcal{T}(dx_1|z_0',u)\nonumber\\
&= \int_{\mathds{X}} f_m(x_1) \left(Q(I_+|x_1)-Q(I_-|x_1)\right)\mathcal{T}(dx_1|z_0',u)\nonumber\\
&\qquad- \int_{\mathds{X}} f_m(x_1) \left(Q(I_+|x_1)-Q(I_-|x_1)\right)\mathcal{T}(dx_1|z_0,u)\label{bnd_2}\\
&\leq (1+\alpha_{\mathds{X}})\rho_{BL}(z_0',z_0).\nonumber
\end{align}
Hence, we can write that
\begin{align*}
& \int_{\mathds{Y}} \bigg| \int_{\mathds{X}} f_m(x_1) z_1(z_0',u,y_1)(dx_1)  -\int_{\mathds{X}} f_m(x_1) z_1(z_0,u,y_1)(dx_1)\bigg|P(dy_1|z_0,u)\\
&\leq 2(1+\alpha_{\mathds{X}})\rho_{BL}(z_0',z_0)
\end{align*}

Now we go back to the term (\ref{second_term}):
\begin{align*}
& \sup_{\|f\|_{BL}\leq1}\int_{\mathds{Y}} \big| f(z_1(z_0',u,y_1))-f(z_1(z_0,u,y_1)) \big| P(dy_1|z_0,u)\\
&\leq  \int_{\mathds{Y}} \rho(z_1(z_0',u,y_1),z_1(z_0,u,y_1))P(dy_1|z_0,u) \\
&= \int_{\mathds{Y}} \sum_{m=1}^{\infty} 2^{-m+1} \bigg| \int_{\mathds{X}} f_m(x_1) z_1(z_0',u,y_1)(dx_1) - \int_{\mathds{X}} f_m(x_1) z_1(z_0,u,y_1)(dx_1)\bigg|P(dy_1|z_0,u) \\
&=  \sum_{m=1}^{\infty} 2^{-m+1} \int_{\mathds{Y}} \bigg| \int_{\mathds{X}} f_m(x_1) z_1(z_0',u,y_1)(dx_1) \nonumber - \int_{\mathds{X}} f_m(x_1) z_1(z_0,u,y_1)(dx_1)\bigg|P(dy_1|z_0,u)\\
&\leq 2(1+\alpha_{\mathds{X}})\rho_{BL}(z_0',z_0).
\end{align*}
Thus, combining all the results:
\begin{align*}
&\sup_{\|f\|_{BL}\leq1}\bigg|\int_{{\cal P}(\mathds{X})} f(z_1)\eta(dz_1|z_0',u)- \int_{{\cal P}(\mathds{X})} f(z_1)\eta(dz_1|z_0,u)\bigg|\nonumber\\
&\leq \|P(\cdot|z_0',u) - P(\cdot|z_0,u)\|_{TV}\\
&\qquad+ \sup_{\|f\|_{BL}\leq1}\int_{\mathds{Y}} \big| f(z_1(z_0',u,y_1)) -f(z_1(z_0,u,y_1)) \big|P(dy_1|z_0,u)\\
&\leq 3(1+\alpha_{\mathds{X}})\rho_{BL}(z_0',z_0).
\end{align*}

We now prove part (ii); we note first that if $\|P(\cdot|z_0',u) - P(\cdot|z_0,u)\|_{TV}\leq 2(1-\delta(Q))\rho_{BL}(z_0',z_0)$, the result follows. Hence, we write  \citep[by][]{dobrushin1956central}
\begin{align*}
&\|P(\cdot|z_0',u) - P(\cdot|z_0,u)\|_{TV}=\sup_{\|f\|_\infty\leq 1}|\int f(y_1)P(dy_1|z_0',u) - \int f(y_1)P(dy_1|z_0,u)|\\
&=\sup_{\|f\|_\infty\leq 1} \bigg|\int f(y_1)Q(dy_1|x_1)\mathcal{T}(dx_1|z_0',u) - \int f(y_1) Q(dy_1|x_1)\mathcal{T}(dx_1|z_0,u) \bigg|\\
&\leq (1-\delta(Q))\|\mathcal{T}(dx_1|z_0',u) -\mathcal{T}(dx_1|z_0,u) \|_{TV}\leq (1-\delta(Q))(1+\alpha_{\mathds{X}})\rho_{BL}(z_0',z_0).
\end{align*}

To prove part (iii) of the theorem, we start from the terms (\ref{first_term}) and (\ref{second_term}). For (\ref{first_term}), we have
\begin{align*}
&\|P(\cdot|z_0',u) - P(\cdot|z_0,u)\|_{TV}\\
&=\sup_{\|f\|_\infty\leq 1} \bigg|\int f(y_1)Q(dy_1|x_1)\mathcal{T}(dx_1|x_0,u)z_0'(dx_0) - \int f(y_1)Q(dy_1|x_1)\mathcal{T}(dx_1|x_0,u)z_0(dx_0) \bigg|\\
&\leq \sup_{\|f\|_\infty\leq 1}\left|\int f(x_0)z_0(dx_0)-\int f(x_0)z_0'(dx_0)\right|=\|z_0-z_0'\|_{TV}.
\end{align*}

For (\ref{second_term}), as before, we start by writing 
\begin{align*}
& \sup_{\|f\|_{BL}\leq1}\int_{\mathds{Y}} \big| f(z_1(z_0',u,y_1))-f(z_1(z_0,u,y_1)) \big| P(dy_1|z_0,u)\\
&\leq   \sum_{m=1}^{\infty} 2^{-m+1} \int_{\mathds{Y}} \bigg| \int_{\mathds{X}} f_m(x_1) z_1(z_0',u,y_1)(dx_1) \nonumber \\
&\phantom{xxxxxxxxxxxxxxxxx}- \int_{\mathds{X}} f_m(x_1) z_1(z_0,u,y_1)(dx_1)\bigg|P(dy_1|z_0,u).
\end{align*}
For each $m$, using the bounding terms (\ref{bnd_1}) and (\ref{bnd_2}), we have 
\begin{align*}
 &\int_{\mathds{Y}} \bigg| \int_{\mathds{X}} f_m(x_1) z_1(z_0',u,y_1)(dx_1) - \int_{\mathds{X}} f_m(x_1) z_1(z_0,u,y_1)(dx_1)\bigg|P(dy_1|z_0,u)\\
&\leq \|P(\cdot|z_0',u) - P(\cdot|z_0,u)\|_{TV}+\|\mathcal{T}(\cdot|z_0',u) - \mathcal{T}(\cdot|z_0,u)\|_{TV}\leq 2 \|z_0'-z_0\|_{TV}
\end{align*}
which completes the proof of part (iii).

For part (iv),  for (\ref{first_term}) we have \citep[by][]{dobrushin1956central}
\begin{align*}
&\|P(\cdot|z_0',u) - P(\cdot|z_0,u)\|_{TV}\\
&=\sup_{\|f\|_\infty\leq 1} \bigg|\int f(y_1)Q(dy_1|x_1)\mathcal{T}(dx_1|z_0',u) - \int f(y_1)Q(dy_1|x_1)\mathcal{T}(dx_1|z_0,u) \bigg|\\
&\leq (1-\delta(Q))\|\mathcal{T}(dx_1|z_0',u)-\mathcal{T}(dx_1|z_0,u)\|_{TV}\\
&= (1-\delta(Q)) \sup_{\|f\|_\infty\leq 1} \bigg|\int f(x_1)\mathcal{T}(dx_1|x_0,u)z_0'(dx_0)-\int f(x_1)\mathcal{T}(dx_1|x_0,u)z_0'(dx_0)\bigg|\\
&\leq (1-\delta(Q))(1-\tilde{\delta}(\mathcal{T}))\|z_0'-z_0\|_{TV}.
\end{align*}
For (\ref{second_term}),  \citep[using Lemma 3.2][]{mcdonald2020exponential} and the analysis above,
\begin{align*}
 &\int_{\mathds{Y}} \bigg| \int_{\mathds{X}} f_m(x_1) z_1(z_0',u,y_1)(dx_1) - \int_{\mathds{X}} f_m(x_1) z_1(z_0,u,y_1)(dx_1)\bigg|P(dy_1|z_0,u)\\
&\leq (2-\delta(Q))\|\mathcal{T}(\cdot|z_0',u)-\mathcal{T}(\cdot|z_0,u)\|_{TV}\leq (2-\delta(Q))(1-\tilde{\delta}(\mathcal{T}))\|z_0'-z_0\|_{TV}
\end{align*}which completes the proof.

\subsection{Proof of Theorem \ref{finite_MDP}}\label{proofThmMain}

Before presenting our proof program and the series of technical results needed, we introduce some notation.

We denote the loss function due to the quantization by $L(z)$, i.e. 
\begin{align*}
L(z)=\rho_{BL}(z,F(z))
\end{align*}
 with $F$ defined in Equation \ref{quant_map}.

In the following, to specify the probability measures according to which the expectations are taken, we use the following notation; for the kernel $\eta^N$ and a policy $\gamma$, we use $E_N^\gamma$ and for the kernel $\eta$ and a policy $\gamma$, we will use $E^\gamma$.

The last notation we introduce is the following: Recall that we denote the optimal cost for the finite model by $J^N_\beta$ and the optimal policy by $\gamma_N^*$. These are defined on a finite set $\mathcal{Z}_{\hat{\pi}}^N$, however, we can always extend them over the original state space $\mathcal{Z}$ so that they are constant over the quantization bins. We denote the extended versions by $\tilde{J}^N_\beta$ and $\tilde{\gamma}_N^*$.

Now, we introduce our value iteration approach for the original model and the finite model. We write for any $k<\infty$ 
\begin{align*}
v_{k+1}(z)&=\min_{u}\left(\tilde{c}(z,u)+\beta\int v_k(z_1)\eta(dz_1|z,u)\right)\quad \forall z\in \mathcal{Z},\\
v^N_{k+1}(z)&=\min_{u}\left(c^N(z,u)+\beta\sum_{z_1} v^N_k(z_1)\eta^N(z_1|z,u)\right)\quad \forall z\in \mathcal{Z}_{\hat{\pi}}^N.
\end{align*}
 where $v_0,v^N_0\equiv 0$. It is well known that the above operations define contractions under either model and hence the value functions converge to the optimal expected discounted cost. In particular, we have that
\begin{align*}
\big|J^N_\beta(z)-v^N_k(z)\big|\leq \|\tilde{c}\|_\infty\frac{\beta^k}{1-\beta},\quad \big|J^*_\beta(z)-v_k(z)\big|\leq \|\tilde{c}\|_\infty\frac{\beta^k}{1-\beta}.
\end{align*}
Notice that if we extend $v^N_{k+1}(z)$ and $c^N(z,u)$ over all of the state space then $\tilde{v}^N_{k+1}(z)$ and $\tilde{c}^N(z,u)$ becomes constant over the quantization bins. Then, the extended value functions for the finite model can be also obtained with
\begin{align}\label{extended_value_it}
\tilde{v}^N_{k+1}(z)&=\min_{u}\left(\tilde{c}(F(z),u)+\beta\int \tilde{v}^N_k(z_1)\eta(dz_1|F(z),u)\right)\quad \forall z\in \mathcal{Z}.
\end{align}
To see this, observe the following:
\begin{align*}
\tilde{v}^N_{k+1}(z)&=\min_{u}\left(c(F(z),u)+\beta\int \tilde{v}^N_k(z_1)\eta(dz_1|F(z),u)\right)\\
&=\min_{u}\left(\tilde{c}(F(z),u)+\beta\sum_{i=1}^{|\mathcal{Z}_{\hat{\pi}}^N|}\int_{Z^N_i} \tilde{v}^N_k(z_1)\eta(dz_1|F(z),u)\right)\\
&=\min_{u}\left(\tilde{c}(F(z),u)+\beta\sum_{i=1}^{|\mathcal{Z}_{\hat{\pi}}^N|}\tilde{v}^N_k(z_1^i)\int_{\mathcal{Z}^N_i} \eta(dy|F(z),u)\right)\\
&=\min_{u}\left(\tilde{c}(F(z),u)+\beta\sum_{i=1}^{|\mathcal{Z}_{\hat{\pi}}^N|}\tilde{v}^N_k(z_1^i) \eta(\mathcal{Z}^N_i|F(z),u)\right)
\end{align*}
where $\mathcal{Z}^n_i$ denotes the ith quantization bin and $\tilde{v}^N_k(z_1^i)$ denotes the value of $\tilde{v}^N_k$ at that quantization bin and it is constant over the bin.


For the proof the goal is to bound the term $|J_\beta(\tilde{\gamma}_N^*,z)-J_\beta(\gamma^*,z)|$. We will see in the following that bounding this term is related to studying the term $|\tilde{J}^N_\beta(\tilde{\gamma}_N^*,z)-J_\beta(\gamma^*,z)|$ (the value function approximation error). Therefore, in what follows, we first analyze $|\tilde{J}^N_\beta(\tilde{\gamma}_N^*,z)-J_\beta(\gamma^*,z)|$ and observe that it can be bounded using the expected loss terms $E[L(Z_t)]$. Finally, we will observe that upper bounds on these expressions can be written through the filter stability term 
$$\sup_{\pi\in\P(\mathds{X})}\sup_{\gamma\in\Gamma}E_{\pi}^{\gamma}\left[\rho_{BL}\left(P^{\pi}(X_N\in\cdot|Y_{[0,N]},\gamma(Y_{[0,N-1]})),P^{\hat{\pi}}(X_N\in\cdot|Y_{[0,N]}),\gamma(Y_{[0,N-1]})\right)\right].$$

Next, we present some supporting technical results.

\begin{lemma}\label{cont_bound}
Under Assumption \ref{belief_reg}, we have
\begin{align*}
&|\tilde{J}^N_\beta(\tilde{\gamma}_N^*,z)-J_\beta(\gamma^*,z)| \\
& \leq \lim_{k\to\infty}\sup_{\gamma\in\Gamma}\Bigg(\alpha_{\tilde{c}} \sum_{t=0}^{k-1}\beta^{t}\left( E^\gamma_{z}\left[L(Z_t)\right]+ 3 \alpha_{\mathcal{Z}} \sum_{m=0}^{t-1}(4\alpha_{\mathcal{Z}}+1)^mE^\gamma_{z}[L(Z_{t-m-1})]\right)\nonumber \\
&\qquad\qquad+  \alpha_{\mathcal{Z}} \sum_{t=0}^{k-1}\beta^{t+1}\|v_{k-t-1}\|_{BL}\left( E^\gamma_{z}\left[L(Z_t)\right]+ 3 \alpha_{\mathcal{Z}} \sum_{m=0}^{t-1}(4\alpha_{\mathcal{Z}}+1)^mE^\gamma_{z}[L(Z_{t-m-1})]\right)\Bigg).
\end{align*}
where $v_k$ denotes the value function at a time step $k$ that is produced by the value iteration with $v_0 \equiv 0$.
\end{lemma}

\begin{proof}

We follow a value iteration approach to write for any $k<\infty$ 
\begin{align*}
v_{k+1}(z)&=\min_{u}\left(\tilde{c}(z,u)+\beta\int v_k(y)\eta(dy|z,u)\right)\quad \forall z\in \mathcal{Z},\\
\tilde{v}^N_{k+1}(z)&=\min_{u}\left(\tilde{c}(F(z),u)+\beta\int \tilde{v}^N_k(y)\eta(dy|F(z),u)\right)\quad \forall z\in \mathcal{Z}.
\end{align*}
 where $v_0,\tilde{v}^N_0\equiv 0$.

We then write
\begin{align}\label{val_it}
|\tilde{J}_\beta^N(z)-J^*_\beta(z)|\leq |\tilde{J}^N_\beta(z)-\tilde{v}_k^N(z)|+|\tilde{v}_k^N(z)-v_k(z)|+|v_k(z)-J^*_\beta(z)|
\end{align}

For the first and the last term, using the fact that the dynamic programming operator is a contraction, we have that
\begin{align*}
\big|\tilde{J}^N_\beta(z)-\tilde{v}^N_k(z)\big|\leq \|\tilde{c}\|_\infty\frac{\beta^k}{1-\beta},\quad \big|J^*_\beta(z)-v_k(z)\big|\leq \|\tilde{c}\|_\infty\frac{\beta^k}{1-\beta}.
\end{align*}

Now, we focus on the second term $|\tilde{v}_k^N(z)-v_k(z)|$. 

\textbf{Step 1:}
We show in the Appendix, Section \ref{SecStep1App}, that for all $k\geq 1$,
\begin{align*}
& |\tilde{v}_k^N(z)-v_k(z)| \\
& \leq\sup_{\gamma\in\Gamma}\left(\alpha_{\tilde{c}} \sum_{t=0}^{k-1}\beta^{t}E^\gamma_{z,N}\left[L(Z_t)\right]+\sum_{t=0}^{k-1}\beta^{t+1}\|v_{k-t-1}\|_{BL}E^\gamma_{z,N}\left[L(Z_t)\right]\alpha_{\mathcal{Z}} \right)
\end{align*}
where $L(z)$ is the loss function due to the quantization, i.e. $L(z)=\rho_{BL}(z,\hat{z})$ where $\hat{z}$ is the representative state $z$ belongs to. 

\textbf{Step 2:}
We show in the Appendix, Section \ref{SecStep2App}, that
\begin{align*}
&E^\gamma_{z,N}\left[L(Z_t)\right]\leq E^\gamma_{z}\left[L(Z_t)\right]+ 3 \alpha_{\mathcal{Z}} \sum_{m=0}^{t-1}(4\alpha_{\mathcal{Z}}+1)^mE^\gamma_{z}[L(Z_{t-m-1})].
\end{align*}

\textbf{Step 3:}
Combining our results, we have that 
\begin{align}\label{value_bound}
&|\tilde{v}_k^N(z)-v_k(z)| \nonumber \\
& \leq\sup_{\gamma}\Bigg(\alpha_{\tilde{c}} \sum_{t=0}^{k-1}\beta^{t}E^\gamma_{z,N}\left[L(Z_t)\right]+\sum_{t=0}^{k-1}\beta^{t+1}\|v_{k-t-1}\|_{BL}E^\gamma_{z,N}\left[L(Z_t)\right]\alpha_{\mathcal{Z}}\Bigg)\nonumber\\
& \leq  \sup_{\gamma}\Bigg(\alpha_{\tilde{c}} \sum_{t=0}^{k-1}\beta^{t}\left( E^\gamma_{z}\left[L(Z_t)\right]+ 3 \alpha_{\mathcal{Z}} \sum_{m=0}^{t-1}(4\alpha_{\mathcal{Z}}+1)^mE^\gamma_{z}[L(Z_{t-m-1})]\right)\nonumber \\
&\qquad\qquad+  \alpha_{\mathcal{Z}} \sum_{t=0}^{k-1}\beta^{t+1}\|v_{k-t-1}\|_{BL}\left( E^\gamma_{z}\left[L(Z_t)\right]+ 3 \alpha_{\mathcal{Z}} \sum_{m=0}^{t-1}(4\alpha_{\mathcal{Z}}+1)^mE^\gamma_{z}[L(Z_{t-m-1})]\right)\Bigg).
\end{align}
Hence, using (\ref{val_it}) and (\ref{value_bound}), we write
\begin{align*}
&|\tilde{J}_\beta^N(z)-J^*_\beta(z)|\leq \lim_{k\to\infty}|\tilde{v}_k^N(z)-v_k(z)|\\
&\leq \lim_{k\to\infty}\sup_{\gamma}\Bigg(\alpha_{\tilde{c}} \sum_{t=0}^{k-1}\beta^{t}\left( E^\gamma_{z}\left[L(Z_t)\right]+ 3 \alpha_{\mathcal{Z}} \sum_{m=0}^{t-1}(4\alpha_{\mathcal{Z}}+1)^mE^\gamma_{z}[L(Z_{t-m-1})]\right)\nonumber \\
&\qquad\qquad+  \alpha_{\mathcal{Z}} \sum_{t=0}^{k-1}\beta^{t+1}\|v_{k-t-1}\|_{BL}\left( E^\gamma_{z}\left[L(Z_t)\right]+ 3 \alpha_{\mathcal{Z}} \sum_{m=0}^{t-1}(4\alpha_{\mathcal{Z}}+1)^mE^\gamma_{z}[L(Z_{t-m-1})]\right)\Bigg).
\end{align*}
\end{proof}

The next result, gives a bound on the loss function occurring from the quantization (on $L(z)$) and relates the bound to the filter stability problem.

\begin{lemma}\label{quant_bound}
For $Z_0=P^{\pi_0}(X_N\in\cdot|Y_{[0,N]},\gamma(Y_{[0,N-1]}))$, where $\pi_0\in\P(\mathds{X})$ is the prior distribution of $X_0$, we have that for any $N<t<\infty$
\begin{align*}
&\sup_{\gamma\in\Gamma}E_{\pi_0}^{\gamma}\left[E_{Z_0}^{\gamma}\left[L(Z_{t})\right]|Y_{[0,N]},\gamma\left(Y_{[0,N-1]}\right)\right]\\
&\leq\sup_{\pi\in\P(\mathds{X})}\sup_{\gamma\in\Gamma}E_{\pi}^{\gamma}\left[\rho_{BL}\left(P^{\pi}(X_N\in\cdot|Y_{[0,N]},\gamma(Y_{[0,N-1]})),P^{\hat{\pi}}(X_N\in\cdot|Y_{[0,N]}),\gamma(Y_{[0,N-1]})\right)\right].
\end{align*}

\end{lemma}
\begin{proof}
We first note that, since the quantizer is a nearest-neighbor quantization, the quantization error is almost surely upper bounded by \newline $\rho_{BL}\left(P^{\pi_{t_-}}(X_{N+t}\in\cdot|Y_{[t,t+N]},U_{[t,t+N-1]}),P^{\hat{\pi}}(X_{N+t}\in\cdot|Y_{[t,t+N]},U_{[t,t+N-1]})\right)$. Using this and the law of total expectation:
\begin{align*}
&E^{\gamma}_{\pi_0}\left[E_{z_0}^{\gamma}\left[L(Z_{t})\right]|Y_{[0,N]},\gamma(Y_{[0,N-1]})\right]\\
&\leq \sup_{\gamma\in\Gamma}E_{\pi_0}^\gamma\bigg[E_{\pi_{t_-}}^\gamma \bigg[\rho_{BL}\bigg(P^{\pi_{t_-}}(X_{N+t}\in\cdot|Y_{[t,t+N]},U_{[t,t+N-1]}) \\
& \qquad \qquad\qquad\qquad\qquad\qquad\qquad ,P^{\hat{\pi}}(X_{N+t}\in\cdot|Y_{[t,t+N]},U_{[t,t+N-1]})\bigg)\bigg] \bigg| {(U,Y)}_{[0,t-1]}\bigg]\nonumber\\
&\leq \sup_{\pi\in\P(\mathds{X})}\sup_{\gamma\in\Gamma}E_{\pi}^{\gamma}\left[\rho_{BL}\left(P^{\pi}(X_N\in\cdot|Y_{[0,N]},U_{[0,N-1]}),P^{\hat{\pi}}(X_N\in\cdot|Y_{[0,N]},U_{[0,N-1]})\right)\right]
\end{align*}
where ${(U,Y)}_{[0,t-1]}=U_{[0,t-1]},Y_{[0,t-1]}$.
\end{proof}

\begin{lemma}\label{val_bound}
Under Assumption \ref{belief_reg} we have that
\begin{align*}
\|v_k\|_{BL}\leq \alpha_{\tilde{c}} \sum_{t=0}^{k-1}(\beta\alpha_{\mathcal{Z}})^t + \|\tilde{c}\|_\infty\sum_{t=0}^{k-1}(\beta\alpha_{\mathcal{Z}})^t\frac{1-\beta^{k-t}}{1-\beta}.
\end{align*}
In particular, we have that for all $k$
\begin{align*}
\|v_k\|_{BL}\leq \frac{1}{1-\beta\alpha_{\mathcal{Z}}}\left(\frac{\|\tilde{c}\|_\infty}{1-\beta}+\alpha_{\tilde{c}}\right).
\end{align*}
\end{lemma}
\begin{proof}
It is easy to see that $\|v_k\|_\infty\leq \|\tilde{c}\|_\infty \sum_{t=0}^{k-1}\beta^t=\|\tilde{c}\|_\infty \frac{1-\beta^k}{1-\beta}$. Then, we use an inductive approach and assume the claim holds for $k$ and analyze the term $k+1$
\begin{align*}
\|v_{k+1}\|_{BL}= \|v_{k+1}\|_\infty +\sup_{x\neq y}\frac{|v_{k+1}(x)-v_{k+1}(y)|}{\rho_{BL}(x,y)}
\end{align*}
For the second term, we have
\begin{align*}
&|v_{k+1}(x)-v_{k+1}(y)|\leq \sup_u\left(|\tilde{c}(x,u)-\tilde{c}(y,u)|+\beta\left|\int v_k(z)\eta(dz|x,u)- \int v_k(z)\eta(dz|y,u)\right|\right)\\
&\leq \alpha_{\tilde{c}}\rho_{BL}(x,y)+\beta \|v_k\|_{BL}\alpha_{\mathcal{Z}} \rho_{BL}(x,y).
\end{align*}
Hence, using the induction hypothesis, we have that
\begin{align*}
&\|v_{k+1}\|_{BL}\leq \|v_{k+1}\|_\infty +\alpha_{\tilde{c}} +\beta \alpha_{\mathcal{Z}}\|v_k\|_{BL} \\
&\qquad\leq\|\tilde{c}\|_\infty \frac{1-\beta^{k+1}}{1-\beta}+\alpha_{\tilde{c}} + \beta\alpha_{\mathcal{Z}}\left(\alpha_{\tilde{c}} \sum_{t=0}^{k-1}(\beta\alpha_{\mathcal{Z}})^t + \|\tilde{c}\|_\infty\sum_{t=0}^{k-1}(\beta\alpha_{\mathcal{Z}})^t\frac{1-\beta^{k-t}}{1-\beta}\right)\\
&\qquad =\alpha_{\tilde{c}} \sum_{t'=0}^{k}(\beta\alpha_{\mathcal{Z}})^{t'} + \|\tilde{c}\|_\infty\sum_{t'=0}^{k}(\beta\alpha_{\mathcal{Z}})^{t'}\frac{1-\beta^{k-t'}}{1-\beta},\quad t'=t+1,
\end{align*}
which concludes the induction argument. Therefore, we can write
\begin{align*}
\|v_k\|_{BL}&\leq \alpha_{\tilde{c}} \sum_{t=0}^{k-1}(\beta\alpha_{\mathcal{Z}})^t + \|\tilde{c}\|_\infty\sum_{t=0}^{k-1}(\beta\alpha_{\mathcal{Z}})^t\frac{1-\beta^{k-t}}{1-\beta}\\
&=\left(\alpha_{\tilde{c}}+\frac{\|\tilde{c}\|_\infty}{1-\beta}\right)\frac{1-(\beta\alpha_{\mathcal{Z}})^k}{1-\beta\alpha_{\mathcal{Z}}}-\frac{\|\tilde{c}\|_\infty\beta^k(1-\alpha_{\mathcal{Z}}^k)}{(1-\beta)(1-\alpha_{\mathcal{Z}})}.
\end{align*}
\end{proof}
Now, we can prove Theorem \ref{finite_MDP}.

{\it Proof of Theorem \ref{finite_MDP}}

Consider the following dynamic programming operators, $\hat{T}_n,T_n: \B_b(\mathcal{Z}) \to \B_b(\mathcal{Z})$ ($\B_b(\mathcal{Z})$ denotes the set of measurable and bounded functions on $\mathcal{Z}$) such that for any $f\in \B_b(\mathcal{Z})$ and for any $z\in \mathcal{Z}$
\begin{align*}
(T_n(f))(z)&:=c(z,\tilde{\gamma}^*_N(z))+\beta \int f(y)\eta(dy|z,\tilde{\gamma}_N^*(z))\\
(\hat{T}_n(f))(z)&:=c(F(z),\tilde{\gamma}_N^*(z))+\beta \int f(y)\eta(dy|F(z),\tilde{\gamma}_N^*(z))
\end{align*}
where $\tilde{\gamma}_N^*$ denotes the extension of the optimal policy for the finite state space model to the belief space $\mathcal{Z}$, i.e. it is constant over the quantization bins, furthermore $F(z)$ is the nearest neighbor quantization map such that it maps $z$ to the closest point from the finite state space $\mathcal{Z}_{\hat{\pi}}^N$ (recall that $\rho_{BL}$ is used to metrize the belief space).

The optimal cost for the finite model, $J^N_\beta$ is only defined on a finite set, we denote by $\tilde{J}^N_\beta$ the extension of the optimal cost for the finite model to the belief space $\P(\mathds{X})$ by assigning the same values over the quantizaton bins , i.e. it is a piece-wise constant function over the quantization bins. Note that, we have 
\begin{align*}
J_\beta(\tilde{\gamma}_N^*,z)&=T_n(J_\beta(\tilde{\gamma}^*_N,z)),\\
\tilde{J}^N_\beta(\tilde{\gamma}_N^*,z) &= \hat{T}_n(\tilde{J}^N_\beta(\tilde{\gamma}_N^*,z)).
\end{align*}

Using these equalities, we write
\begin{align}\label{robust_bound}
&J_\beta(\tilde{\gamma}_N^*,z)-J_\beta(\gamma^*,z)\leq |T_n(J_\beta(\tilde{\gamma}_N^*,z))-T_n(J_\beta(\gamma^*,z))|\nonumber\\
&\hspace{3.5cm}+|T_n(J_\beta(\gamma^*,z))-\hat{T}_n(J_\beta(\gamma^*,z))|\nonumber\\
&\hspace{3.5cm}+|\hat{T}_n(J_\beta(\gamma^*,z))-\hat{T}_n(\tilde{J}^N_\beta(\tilde{\gamma}_N^*,z))|\nonumber\\
&\hspace{3.5cm}+|\tilde{J}^N_\beta(\tilde{\gamma}_N^*,z)-J_\beta(\gamma^*,z)|\nonumber\\
&\leq \beta \int |J_\beta(\tilde{\gamma}_N^*,z_1)-J_\beta(\gamma^*,z_1)|\eta(dz_1|z,\tilde{\gamma}_N^*(z))\nonumber\\
&\qquad+|c(z,\tilde{\gamma}_N^*(z))-c(F(z),\tilde{\gamma}_N^*(z))| \nonumber\\
&\qquad
+\beta|\int J_\beta(\gamma^*,z_1)\eta(dz_1|z,\tilde{\gamma}_N^*(z))-\int J_\beta(\gamma^*,z_1)\eta(dz_1|F(z),\tilde{\gamma}_N^*(z))|\nonumber\\
&\qquad+\beta \int |\tilde{J}_\beta^N(\tilde{\gamma}_N^*,z_1)-J_\beta(\gamma^*,z_1)|\eta(dz_1|F(z),\tilde{\gamma}_N^*(z))\nonumber\\
&\qquad+|\tilde{J}_\beta^N(\tilde{\gamma}_N^*,z)-J_\beta(\gamma^*,z)|\nonumber\\
&\leq \beta\sup_\gamma\bigg( E^\gamma_{z}[|J_\beta(\tilde{\gamma}_N^*,Z_1)-J_\beta(\gamma^*,Z_1)|]\bigg)+ \alpha_c L(z)+\beta\alpha_{\mathcal{Z}} L(z)\|J_\beta^*\|_{BL} \nonumber\\
&\qquad\qquad\qquad +\beta E^\gamma_{z,N}[|\tilde{J}_\beta^N(\tilde{\gamma}_N^*,Z_1)-J_\beta(\gamma^*,Z_1)|]+|\tilde{J}_\beta^N(\tilde{\gamma}_N^*,z)-J_\beta(\gamma^*,z)|
\end{align}
where $E^N_{z,\gamma}$ denotes the expectation with respect the the kernel $\eta(dy|F(z),\gamma(z))$ when the initial point is $z$.

Now, using Lemma \ref{cont_bound}, we can write that:
\begin{align}\label{cont_g_notation}
& |\tilde{J}^N_\beta(\tilde{\gamma}_N^*,z)-J_\beta(\gamma^*,z)| \nonumber \\ 
&\leq \lim_{k\to\infty}\sup_{\gamma}\Bigg( \alpha_c \sum_{t=0}^{k-1}\beta^{t}\left(E^\gamma_{z}\left[L(Z_t)\right]+ 3 \alpha_{\mathcal{Z}} \sum_{m=0}^{t-1}(4\alpha_{\mathcal{Z}}+1)^mE^\gamma_{z}[L(Z_{t-m-1})]\right)\nonumber \\
&+  \alpha_{\mathcal{Z}} \sum_{t=0}^{k-1}\beta^{t+1}\|v_{k-t-1}\|_{BL}\left(E^\gamma_{z}\left[L(Z_t)\right]+ 3 \alpha_{\mathcal{Z}} \sum_{m=0}^{t-1}(4\alpha_{\mathcal{Z}}+1)^mE^\gamma_{z}[L(Z_{t-m-1})]\right)\Bigg) \nonumber \\
& :=g(z).
\end{align}
We first introduce the following notation along with the above notation (\ref{cont_g_notation})
\begin{align*}
f(z)&:=|J_\beta(\tilde{\gamma}_N^*,z)-J_\beta(\gamma^*,z)|.
\end{align*}

Notice that with the new notation, we can rewrite the bound on (\ref{robust_bound}) as:
\begin{align*}
f(z)&\leq \beta \sup_\gamma E_z^\gamma[f(Z_1)]+\alpha_cL(z)+\beta \alpha_{\mathcal{Z}} L(z)\|J^*_\beta\|_{BL}+\beta E_{z,N}^\gamma[g(Z_1)]+g(z)\\
&\leq  \beta \sup_\gamma E_z^\gamma[f(Z_1)]+\alpha_cL(z)+\beta \alpha_{\mathcal{Z}} L(z)\|J^*_\beta\|_{BL}+\beta E_{z}^\gamma[g(Z_1)]+\beta\|g\|_{BL}\alpha_{\mathcal{Z}}L(z) +g(z)
\end{align*}
where we used the fact that 
\begin{align*}
 E_{z,N}^\gamma[g(Z_1)]- E_{z}^\gamma[g(Z_1)]&=\int g(z_1)\eta(dz_1|F(z),\gamma(z))-\int g(z_1)\eta(dz_1|z,\gamma(z))\\
&\leq \|g\|_{BL}\alpha_{\mathcal{Z}}L(z).
\end{align*}
Using the same bound on $f(Z_1)$ one can also write that for any initial point $z_0$:
\begin{align*}
f(z_0)&\leq \sup_{\gamma\in\Gamma}\beta^2E^\gamma_{z_0}[f(Z_2)]+ \alpha_c\sum_{t=0}^{1}\beta^tE^\gamma_{z_0}[L(Z_t)]+\alpha_{\mathcal{Z}}\|J_\beta^*\|_{BL}\sum_{t=0}^{1}\beta^{t+1}E^\gamma_{z_0}[L(Z_t)]\nonumber\\
&+\sum_{t=0}^{1}\beta^{t+1}E^\gamma_{z_0}[g(Z_t)]+\|g\|_{BL}\alpha_{\mathcal{Z}}\sum_{t=0}^{1}\beta^{t+1}E^\gamma_{z_0}[L(Z_t)]+\sum_{t=0}^{1}\beta^tE^\gamma_{z_0}[g(Z_t)].
\end{align*}
In general, for any $k<\infty$, we can write
\begin{align}\label{robust_bound2}
f(z_0)&\leq \sup_{\gamma\in\Gamma}\Bigg(\beta^kE^\gamma_{z_0}[f(Z_k)]+ \alpha_c\sum_{t=0}^{k-1}\beta^tE^\gamma_{z_0}[L(Z_t)]+\alpha_{\mathcal{Z}}\|J_\beta^*\|_{BL}\sum_{t=0}^{k-1}\beta^{t+1}E^\gamma_{z_0}[L(Z_t)]\nonumber\\
&+\sum_{t=0}^{k-1}\beta^{t+1}E^\gamma_{z_0}[g(Z_t)]+\|g\|_{BL}\alpha_{\mathcal{Z}}\sum_{t=0}^{k-1}\beta^{t+1}E^\gamma_{z_0}[L(Z_t)]+\sum_{t=0}^{k-1}\beta^tE^\gamma_{z_0}[g(Z_t)]\Bigg).
\end{align}
Recall that our main goal is to bound $E_{\pi_0}[f(Z_0)|Y_{[0,N]},\gamma(Y_{[0,N-1]})]$ where $Z_0=P_{\pi_0}(X_N\in\cdot|Y_{[0,N]},\gamma(Y_{[0,N-1]})$. To this end, first notice that using Lemma \ref{quant_bound}, we have for any $t$
\begin{align}\label{loss_bound}
&E_{\pi_0}\left[E^\gamma_{Z_0}[L(Z_t)]|Y_{[0,N]},\gamma(Y_{[0,N-1]})\right]\nonumber\\
&\leq \sup_{\pi\in\P(\mathds{X})}\sup_{\gamma\in\Gamma}E^{\pi,\gamma}\left[\rho_{BL}\left(P^{\pi}(X_N\in\cdot|Y_{[0,N]}),P^{\hat{\pi}}(X_N\in\cdot|Y_{[0,N]})\right)\right]:=B.
\end{align}

Using this bound and Lemma \ref{val_bound} for $\|v_{k-t-1}\|_{BL}$, we can write that
\begin{align}\label{g_bound}
 &E\left[E^\gamma_{Z_0}[g(Z_t)]\big|Y_{[0,N]},\gamma(Y_{[0,N-1]})\right]\nonumber\\
&\leq  \lim_{k\to\infty}\Bigg(\alpha_c \sum_{t=0}^{k-1}\beta^{t}\left( B+ 3\alpha_{\mathcal{Z}} \sum_{m=0}^{t-1}B(4\alpha_{\mathcal{Z}}+1)^m\right) \nonumber \\
& \qquad +  \alpha_{\mathcal{Z}} \sum_{t=0}^{k-1}\beta^{t+1}\|v_{k-t-1}\|_{BL}\left(B+ 3\alpha_{\mathcal{Z}} \sum_{m=0}^{t-1}B(4\alpha_{\mathcal{Z}}+1)^m\right)\Bigg)\nonumber\\
&\leq B \left(\alpha_c + \beta\alpha_{\mathcal{Z}}\|J_\beta^*\|_{BL}\right)\lim_{k\to\infty}\sum_{t=0}^{k-1}\beta^{t}\left( 1+ 3\alpha_{\mathcal{Z}} \sum_{m=0}^{t-1}(4\alpha_{\mathcal{Z}}+1)^m\right)\nonumber\\
&\leq B \left(\alpha_c + \beta\alpha_{\mathcal{Z}}\|J_\beta^*\|_{BL}\right)\lim_{k\to\infty}\sum_{t=0}^{k-1}\beta^{t}\left(1+3\alpha_{\mathcal{Z}}\frac{(4\alpha_{\mathcal{Z}}+1)^t-1}{4\alpha_{\mathcal{Z}}}\right)\nonumber\\
&\leq B \left(\alpha_c + \beta\alpha_{\mathcal{Z}}\|J_\beta^*\|_{BL}\right)\lim_{k\to\infty}\sum_{t=0}^{k-1}\beta^{t}(4\alpha_{\mathcal{Z}}+1)^t\nonumber\\
&=  B K_0(\beta,\alpha_{\mathcal{Z}},\alpha_{\tilde{c}},\|\tilde{c}\|_\infty)
\end{align}
where
\begin{align*}
K_0(\beta,\alpha_{\mathcal{Z}},\alpha_{\tilde{c}},\|\tilde{c}\|_\infty) =&  \left(\alpha_c + \beta\alpha_{\mathcal{Z}}\|J_\beta^*\|_{BL}\right)\frac{1}{1-\beta(4\alpha_{\mathcal{Z}}+1)}.
\end{align*}

By Lemma \ref{g_lb_bound} we also have that
\begin{align*}
\|g\|_{BL}\leq   &(\alpha_c+\beta\alpha_{\mathcal{Z}}\|J_\beta\|_{BL})\left(\frac{2}{1-\beta(4\alpha_{\mathcal{Z}}+1)}+\frac{3\alpha_{\mathcal{Z}}}{1-\beta\alpha_{\mathcal{Z}}}+\frac{9\alpha_{\mathcal{Z}}^2}{(1-\beta(4\alpha_{\mathcal{Z}}+1))^2}\right)\\
&:=\hat{K}_0(\beta,\alpha_{\mathcal{Z}},\alpha_{\tilde{c}},\|\tilde{c}\|_\infty) 
\end{align*}


 Going back to (\ref{robust_bound2}), using (\ref{loss_bound}) and (\ref{g_bound}) and taking the limit $k\to\infty$:
\begin{align*}
E[f(Z_0)|Y_{[0,N]},\gamma(Y_{[0,N-1]})]&\leq \frac{B\alpha_c}{1-\beta}+\frac{\beta \alpha_{\mathcal{Z}}\|J_\beta^*\|_{BL}B}{1-\beta}+\frac{\beta  K_0(\beta,\alpha_{\mathcal{Z}},\alpha_{\tilde{c}},\|\tilde{c}\|_\infty) B}{1-\beta}\\
&\qquad+\frac{ \hat{K}_0(\beta,\alpha_{\mathcal{Z}},\alpha_{\tilde{c}},\|\tilde{c}\|_\infty)\alpha_{\mathcal{Z}}\beta B}{1-\beta}+\frac{ K_0(\beta,\alpha_{\mathcal{Z}},\alpha_{\tilde{c}},\|\tilde{c}\|_\infty)B}{1-\beta}.
\end{align*}

Thus, we can write that
\begin{align*}
E[f(Z_0)|&Y_{[0,N]},\gamma(Y_{[0,N-1]})]\leq \\
&K \sup_{\pi\in\P(\mathds{X})}\sup_{\gamma\in\Gamma}E^{\pi,\gamma}\left[\rho_{BL}\left(P^{\pi}(X_N\in\cdot|Y_{[0,N]}),P^{\hat{\pi}}(X_N\in\cdot|Y_{[0,N]})\right)\right],
\end{align*}
where
\begin{align}\label{KDefinition1}
K=\frac{\alpha_c+\beta\alpha_{\mathcal{Z}}\|J^*_\beta\|_{BL}+(\beta+1) K_0(\beta,\alpha_{\mathcal{Z}},\alpha_{\tilde{c}},\|\tilde{c}\|_\infty)+\hat{K}_0(\beta,\alpha_{\mathcal{Z}},\alpha_{\tilde{c}},\|\tilde{c}\|_\infty) \beta\alpha_{\mathcal{Z}} }{1-\beta}
\end{align}
and using Lemma \ref{lip_val}
\begin{align*}
\|J_\beta^*\|_{BL}\leq \frac{1}{1-\beta\alpha_{\mathcal{Z}}}\left(\frac{\|\tilde{c}\|_\infty}{1-\beta}+\alpha_{\tilde{c}}\right).
\end{align*}

\section{Conclusion}
In this paper, we studied performance bounds for policies that use only a finite-window of recent observation and action variables rather than the entire history in partially observed stochastic control problems. We have rigorously established approximation bounds that are easy to compute, and have shown that this bound critically depends on the ergodicity and stability properties of the belief-state process. We have provided the results for continuous-space valued state spaces and finite observation and action spaces, however our studies suggest that these results can also be generalized to real valued observation and actions also. Application to decentralized POMDPs is another direction that will benefit from the analysis presented.

\appendix
\section{Technical Proofs of Supporting Results}
In this section, we prove the proofs of some technical results.
\subsection{Proof of \textbf{Step 1} in the Proof of Lemma \ref{cont_bound}}\label{SecStep1App}
We claim that
\begin{align*}
|\tilde{v}_k^N(z)-v_k(z)|\leq\sup_{\gamma\in\Gamma}\left(\alpha_{\tilde{c}} \sum_{t=0}^{k-1}\beta^{t}E^N_{z,\gamma}\left[L(Z_t)\right]+\sum_{t=0}^{k-1}\beta^{t+1}\|v_{k-t-1}\|_{BL}E^N_{z,\gamma}\left[L(Z_t)\right]\alpha_{\mathcal{Z}} \right)
\end{align*}
where $L(z)$ is the loss function due to the quantization, i.e. $L(z)=\rho_{BL}(z,F(z))$ 
 with $F$ defined in (\ref{quant_map}).

We prove the claim using an inductive approach: for $k=1$ we have (noting $v_0,\tilde{v}^N_0\equiv 0$)
\begin{align*}
&\tilde{v}_1^N(z)=\min_u\left(\tilde{c}(F(z),u)+\beta \int \tilde{v}_0^N(y)\eta(dy|F(z),u)\right)=\min_u \tilde{c}(F(z),u)\\
&v_1(z)=\min_u\left(\tilde{c}(z,u)+\beta \int v_0(y)\eta(dy|z,u)\right)=\min_u \tilde{c}(z,u).
\end{align*} 
Note that under the stated assumptions the measurable selection conditions hold, and the minimum can be achieved using a policy $\gamma$ for the original model and a policy $\gamma^N$ for the finite model which defined only on a finite set. By extending the finite model policy $\gamma^N$ over all state space $\P(\mathds{X})$, we can write that
\begin{align*}
|\tilde{v}_1^N(z)-v_1(z)|&\leq \max\bigg(\tilde{c}(F(z),\gamma(z))-\tilde{c}(z,\gamma(z)),\tilde{c}(z,\gamma^N(z))-\tilde{c}(F(z),\gamma^N(z))\bigg)\\
&\leq\sup_\gamma |\tilde{c}(F(z),\gamma(z))-\tilde{c}(z,\gamma(z))| \leq \alpha_{\tilde{c}} L(z)
\end{align*}
which completes the proof for $k=1$. Now, we assume that the claim holds for $k$ and analyze the step $k+1$. Similar to $k=1$ case, we can write 
\begin{align*}
&|\tilde{v}_{k+1}^N(z)-v_{k+1}(z)|\\
&\leq \sup_\gamma\left|\tilde{c}(F(z),\gamma(z))-\tilde{c}(z,\gamma(z))+ \beta \int \tilde{v}_k^N(y)\eta(dy|F(z),\gamma(z)) -\beta \int v_k(y)\eta(dy|z,\gamma(z))\right|\\
&\leq \sup_\gamma \Bigg(\left|\tilde{c}(F(z),\gamma(z))-\tilde{c}(z,\gamma(z))\right|+ \beta \int |\tilde{v}_k^N(y)-v_k(y)|\eta(dy|F(z),\gamma(z))\\
&\qquad\qquad+\beta\left| \int v_k(y) \eta(dy|F(z),\gamma(z))-\int v_k(y) \eta(dy|z,\gamma(z))\right|\Bigg)\\
&\leq \alpha_{\tilde{c}} L(z) +\beta \|v_k\|_{BL}\alpha_{\mathcal{Z}} L(z)\\
&+\beta \sup_\gamma\bigg(\int \left|\alpha_{\tilde{c}} \sum_{t=0}^{k-1}\beta^{t}E^N_{y,\gamma}\left[L(Z_t)\right]+\sum_{t=0}^{k-1}\beta^{t+1}\|v_{k-t-1}\|_{BL}E^N_{y,\gamma}\left[L(Z_t)\right]\alpha_{\mathcal{Z}}]\right|\eta(dy|F(z),\gamma(z))\bigg)\\
&\leq \sup_\gamma\bigg(\alpha_{\tilde{c}} L(z) + \alpha_{\tilde{c}} \sum_{t=0}^{k-1}\beta^{t+1} E^N_{z,\gamma}\left[L(Z_{t+1})\right]+ \beta \|v_k\|_{BL}\alpha_{\mathcal{Z}}L(z)\\
&\qquad\qquad+\sum_{t=0}^{k-1}\beta^{t+2}\|v_{k-t-1}\|_{BL}E^N_{z,\gamma}[L(Z_{t+1})]\bigg)\\
&=\sup_\gamma\left(\alpha_{\tilde{c}} \sum_{t'=0}^{k}\beta^{t'} E^N_{z,\gamma}\left[L(Z_{t'})\right]+\sum_{t'=0}^{k}\beta^{t'+1}\|v_{k-t'}\|_{BL}E^N_{z,\gamma}\left[L(Z_{t'})\right]\right),\quad (t'=t+1).
\end{align*}
For the last two steps, note that $E_{y,\gamma}^N\left[L(Z_t)\right]$ denotes the expected loss at time $t$ when the initial state $Z_0=y$, thus using the iterative expectation we have
\begin{align*}
\int E_{y,\gamma}^N\left[L(Z_t)\right]\eta(dy|F(z),\gamma(z))=E_{z,\gamma}^N\left[L(Z_{t+1})\right].
\end{align*}

Hence, we have proved that for all $k\geq 1$
\begin{align*}
|\tilde{v}_k^N(z)-v_k(z)|\leq\sup_\gamma\bigg(\alpha_{\tilde{c}} \sum_{t=0}^{k-1}\beta^{t}E^N_{z,\gamma}\left[L(Z_t)\right]&+\sum_{t=0}^{k-1}\beta^{t+1}\|v_{k-t-1}\|_{BL}E^N_{z,\gamma}\left[L(Z_t)\right]\alpha_{\mathcal{Z}}\bigg).
\end{align*}

\subsection{Proof of \textbf{Step 2} in the Proof of Lemma  \ref{cont_bound}} \label{SecStep2App}
The claim is that 
\begin{align*}
&E^N_{z,\gamma}\left[L(Z_t)\right]\leq E_{z,\gamma}\left[L(Z_t)\right]+ 3 \alpha_{\mathcal{Z}} \sum_{m=0}^{t-1}(4\alpha_{\mathcal{Z}}+1)^mE_{z,\gamma}[L(Z_{t-m-1})].
\end{align*}

We first write 
\begin{align}\label{loss_first_bound}
&E^N_{z,\gamma}\left[L(Z_t)\right]\leq E_{z,\gamma}\left[L(Z_t)\right]+\bigg(\|L\|_{BL} \bigg) \rho_{BL}\left(P^{\gamma}_{z,t},P^{\gamma,N}_{z,t}\right)
\end{align}
where \[P^{\gamma}_{z,t}, \quad P^{\gamma,N}_{z,t}\] are the marginal distributions of the state $z_t$ for the true model and approximate model respectively, with $Z_0=z$.

We focus on the term $\|L\|_{BL}\rho_{BL}\left(P^{\gamma}_{z,t},P^{\gamma,N}_{z,t}\right)$. We first claim that $\|L\|_{BL}\leq 3$, where $L(z)=\rho(z,F(z))$. Recall that $F$ determines quantization given in (\ref{quant_map}).
\begin{align*}
\|L\|_{BL}\leq \|L\|_\infty + \sup_{z,z'}\frac{\left|\rho(z,F(z))-\rho(z',F(z'))\right|}{\rho(z,z')}\leq 2+ \sup_{z,z'}\frac{\left|\rho(z,F(z))-\rho(z',F(z'))\right|}{\rho(z,z')}
\end{align*}
where we used the fact that 
\begin{align*}
\|L\|_\infty=\sup_z \rho_{BL}(z,F(z))\leq 2
\end{align*}
as for any two probability measures $\mu,\nu$,we have that $\rho_{BL}(\mu,\nu)\leq 2$ (see (\ref{BLmetric})).

Note that $F$ is a nearest neighbor quantizer as defined in ($\ref{quant_map}$), thus we can write that 
\begin{align*}
\left|\rho(z,F(z))-\rho(z',F(z'))\right|&\leq \max\left(\rho(z,F(z'))-\rho(z',F(z')),\rho(z',F(z))-\rho(z,F(z)) \right)\\
&\leq \sup_{\hat{z}} \left|\rho(z,\hat{z})-\rho(z',\hat{z})\right|\leq \rho(z,z')
\end{align*}
where for the last step we used the triangle inequality. Hence we can conclude that 
\begin{align}\label{L_bl_bound}
\|L\|_{BL}\leq2+ \sup_{z,z'}\frac{\left|\rho(z,F(z))-\rho(z',F(z'))\right|}{\rho(z,z')}\leq 3.
\end{align}

Now, we show that 
\begin{align}\label{product_meas_bl}
\rho_{BL}\left(P^{\gamma}_{z,t},P^{\gamma,N}_{z,t}\right)\leq \alpha_{\mathcal{Z}} \sum_{m=0}^{t-1}(4\alpha_{\mathcal{Z}}+1)^mE_{z,\gamma}[L(Z_{t-m-1})].
\end{align}
We prove the claim by induction: for $t=1$
\begin{align*}
\rho_{BL}\left(P^{\gamma}_{z,1},P^{\gamma,N}_{z,1}\right)&=\sup_{\|f\|_{BL}\leq1}\left|\int f(z_1)\eta(dz_1|F(z),u_1)-\int f(z_1)\eta(z_1)\eta(dz_1|z,u_1)\right|\\
&\leq \alpha_{\mathcal{Z}} L(z).
\end{align*}
Now assume the claim holds for $t-1$
\begin{align*}
&\rho_{BL}\left(P^{\gamma}_{z,t},P^{\gamma,N}_{z,t}\right) =\sup_{\|f\|_{BL}\leq 1}\bigg|\int f(z_t)\eta(dz_t|z_{t-1},\gamma(z_{t-1})))P^{\gamma}_{z,t-1}(dz_{t-1})\\
&\hspace{50mm}-\int f(z_t)\eta(dz_t|F(z_{t-1}),\gamma(z_{t-1}))P^{\gamma,N}_{z,t-1}(dz_{t-1})\bigg|\\
&\leq \sup_{\|f\|_{BL}\leq 1}\bigg|\int f(z_t)\eta(dz_t|z_{t-1},\gamma(z_{t-1}))P^{\gamma}_{z,t-1}(dz_{t-1})\\
&\hspace{25mm}-\int f(z_t)\eta(dz_t|z_{t-1},\gamma(z_{t-1}))P^{\gamma,N}_{z,t-1}(dz_{t-1})\bigg|\\
&\quad+ \sup_{\|f\|_{BL}\leq 1}\bigg|\int f(z_t)\eta(dz_t|z_{t-1},\gamma(z_{t-1}))P^{\gamma,N}_{z,t-1}(dz_{t-1})\\
&\hspace{25mm}-\int f(z_t)\eta(dz_t|F(z_{t-1}),\gamma(z_{t-1}))P^{\gamma,N}_{z,t-1}(dz_{t-1})\bigg|\\
&\leq (\|f\|_\infty+\alpha_{\mathcal{Z}})\rho_{BL}\left(P^{\gamma}_{z,t-1},P^{\gamma,N}_{z,t-1}\right)+\alpha_{\mathcal{Z}}E^N_z[L(Z_{t-1})]\\
&\leq  (1+\alpha_{\mathcal{Z}})\rho_{BL}\left(P^{\gamma}_{z,t-1},P^{\gamma,N}_{z,t-1}\right)+\alpha_{\mathcal{Z}}E^N_z[L(Z_{t-1})].
\end{align*}
Using (\ref{loss_first_bound}), we can write 
\begin{align*}
\rho_{BL}&\left(P^{\gamma}_{z,t},P^{\gamma,N}_{z,t}\right)\\
&\leq (1+\alpha_{\mathcal{Z}})\rho_{BL}\left(P^{\gamma}_{z,t-1},P^{\gamma,N}_{z,t-1}\right)+\alpha_{\mathcal{Z}} \bigg( E_z[L(Z_{t-1})]+3 \rho_{BL}\left(P^{\gamma}_{z,t-1},P_{z,t-1}^{\gamma,N}\right)\bigg)\\
&=(4\alpha_{\mathcal{Z}} +1)\rho_{BL}\left(P^{\gamma}_{z,t-1},P^{\gamma,N}_{z,t-1}\right)+\alpha_{\mathcal{Z}} E_z[L(Z_{t-1})]\\
&\leq (4\alpha_{\mathcal{Z}} +1) \left(\alpha_{\mathcal{Z}} \sum_{m=0}^{t-2}(4\alpha_{\mathcal{Z}}+1)^mE_z[L(Z_{t-m-2})]\right)+\alpha_{\mathcal{Z}} E_z[L(Z_{t-1})]\\
&=\alpha_{\mathcal{Z}} \sum_{m=0}^{t-1}(4\alpha_{\mathcal{Z}}+1)^mE[L_z(Z_{t-m-1})]
\end{align*}
which completes the proof of (\ref{product_meas_bl}).

Now, we go back to the main claim and start from (\ref{loss_first_bound}) to write
\begin{align*}
&E^N_{z,\gamma}\left[L(Z_t)\right]\leq E_{z,\gamma}\left[L(Z_t)\right]+\|L\|_{BL}\rho_{BL}\left(P^{\gamma}_{z,t},P^{\gamma,N}_{z,t}\right)\leq  E_{z,\gamma}\left[L(Z_t)\right]+3\rho_{BL}\left(P^{\gamma}_{z,t},P^{\gamma,N}_{z,t}\right)\\
&\leq E_{z,\gamma}\left[L(Z_t)\right]+ 3 \alpha_{\mathcal{Z}} \sum_{m=0}^{t-1}(4\alpha_{\mathcal{Z}}+1)^mE_{z,\gamma}[L(Z_{t-m-1})].
\end{align*}

\begin{remark}\label{beta_remark}
Note that  for the previous proof, $(4\alpha_{\mathcal{Z}}+1)$ can be replaced by $1+\alpha_{\mathcal{Z}}(\|L\|_\infty+2)$ where 
\begin{align*}
\|L\|_\infty=\sup_{\pi\in\P(\mathds{X})}\sup_{\gamma\in\Gamma}\sup_{y_{[0,N]},u_{[0,N-1]}} \rho_{BL}\left(P^{\pi}(\cdot|y_{[0,N]},u_{[0,N-1]}),P^{\hat{\pi}}(\cdot|y_{[0,N]},u_{[0,N-1]})\right)
\end{align*}
which is upper bounded by $2$, however, usually smaller than $2$, provided there is a uniform filter stability.
\end{remark}

\begin{lemma}\label{action_policy_change}
We introduce the following notation
\begin{align*}
\hat{E}_{z}[L(Z_t)]&=\sup_{u_0}\int\dots\sup_{u_{t-2}}\int \sup_{u_{t-1}}\int L(z_t)\eta(dz_t|z_{t-1},u_{t-1})\eta(dz_{t-1}|z_{t-2},u_{t-2})\dots \eta(dz_1|z,u_0).
\end{align*}
Under Assumption \ref{belief_reg}, we have that
\begin{align*}
\sup_{\gamma\in\Gamma}E_{z,\gamma}[L(Z_t)]= \hat{E}_{z}[L(Z_t)].
\end{align*}
\end{lemma}
\begin{proof}
Recall that for a fixed policy $\gamma=\{\gamma_t\}_t$,
\begin{align*}
E_{z,\gamma}[L(Z_t)]=\int L(z_t)\eta(dz_t|z_{t-1},\gamma_{t-1}(z_{t-1}))\eta(dz_{t-1}|z_{t-1},\gamma_{t-2}(z_{t-2}))\dots\eta(dz_1|z,\gamma_0(z_0))
\end{align*}
It is easy to see that 
\begin{align*}
&\int L(z_t)\eta(dz_t|z_{t-1},\gamma_{t-1}(z_{t-1}))\eta(dz_{t-1}|z_{t-1},\gamma_{t-2}(z_{t-2}))\dots\eta(dz_1|z,\gamma_0(z_0))\\
&\leq\int\sup_{u_{t-1}}\int L(z_t)\eta(dz_t|z_{t-1},u_{t-1})\eta(dz_{t-1}|z_{t-1},\gamma_{t-2}(z_{t-2}))\dots\eta(dz_1|z,\gamma_0(z_0))
\end{align*}
By repeating this step we can have 
\begin{align*}
&\int L(z_t)\eta(dz_t|z_{t-1},\gamma_{t-1}(z_{t-1}))\eta(dz_{t-1}|z_{t-1},\gamma_{t-2}(z_{t-2}))\dots\eta(dz_1|z,\gamma_0(z_0))\\
&\leq\sup_{u_0}\int\dots\sup_{u_{t-2}}\int \sup_{u_{t-1}}\int L(z_t)\eta(dz_t|z_{t-1},u_{t-1})\eta(dz_{t-1}|z_{t-2},u_{t-2})\dots \eta(dz_1|z,u_0)
\end{align*}
Hence by taking supremum over all policies we can write
\begin{align*}
\sup_{\gamma\in\Gamma}E_{z,\gamma}[L(Z_t)]\leq\hat{E}_{z}[L(Z_t)].
\end{align*}
For the other direction, we first focus on the term $\sup_{u_{t-1}}\int L(z_t)\eta(dz_t|z_{t-1},u_{t-1})$, using Assumption \ref{belief_reg}, one can show that there exists a measurable map $\gamma_{t-1}$ such that
\begin{align*}
\sup_{u_{t-1}}\int L(z_t)\eta(dz_t|z_{t-1},u_{t-1})=\int L(z_t)\eta(dz_t|z_{t-1},\gamma_{t-1}(z_{t-1})).
\end{align*}
Using the same argument, we can see that there exists a sequence of measurable functions \newline$\{\gamma_0,\gamma_1,\dots,\gamma_{t-1}\}$ such that
\begin{align*}
&\sup_{u_0}\int\dots\sup_{u_{t-2}}\int \sup_{u_{t-1}}\int L(z_t)\eta(dz_t|z_{t-1},u_{t-1})\eta(dz_{t-1}|z_{t-2},u_{t-2})\dots \eta(dz_1|z,u_0)\\
&=\int L(z_t)\eta(dz_t|z_{t-1},\gamma_{t-1}(z_{t-1}))\eta(dz_{t-1}|z_{t-1},\gamma_{t-2}(z_{t-2}))\dots\eta(dz_1|z,\gamma_0(z_0)).
\end{align*}
Hence we can write that
\begin{align*}
\hat{E}_{z}[L(Z_t)]\leq \sup_{\gamma\in\Gamma}E_{z,\gamma}[L(Z_t)]
\end{align*}
which proves the main claim.
\end{proof}

\begin{lemma}\label{g_lb_bound}
Under Assumption \ref{belief_reg}
\begin{align*}
\|g\|_{BL}\leq   (\alpha_c+\beta\alpha_{\mathcal{Z}}\|J_\beta\|_{BL})\left(\frac{2}{1-\beta(4\alpha_{\mathcal{Z}}+1)}+\frac{3\alpha_{\mathcal{Z}}}{1-\beta\alpha_{\mathcal{Z}}}+\frac{9\alpha_{\mathcal{Z}}^2}{(1-\beta(4\alpha_{\mathcal{Z}}+1))^2}\right).
\end{align*}
\end{lemma}
\begin{proof}
Using Lemma \ref{action_policy_change}, we can write that
\begin{align*}
g(z)&=\lim_{k\to\infty}\sup_{\gamma}\Bigg( \alpha_c \sum_{t=0}^{k-1}\beta^{t}\left(E_{z,\gamma}\left[L(Z_t)\right]+ 3 \alpha_{\mathcal{Z}} \sum_{m=0}^{t-1}(4\alpha_{\mathcal{Z}}+1)^mE_{z,\gamma}[L(Z_{t-m-1})]\right)\nonumber \\
&+  \alpha_{\mathcal{Z}} \sum_{t=0}^{k-1}\beta^{t+1}\|v_{k-t-1}\|_{BL}\left(E_{z,\gamma}\left[L(Z_t)\right]+ 3 \alpha_{\mathcal{Z}} \sum_{m=0}^{t-1}(4\alpha_{\mathcal{Z}}+1)^mE_{z,\gamma}[L(Z_{t-m-1})]\right)\Bigg)\\
&=\lim_{k\to\infty}\Bigg( \alpha_c \sum_{t=0}^{k-1}\beta^{t}\left(\hat{E}_{z}\left[L(Z_t)\right]+ 3 \alpha_{\mathcal{Z}} \sum_{m=0}^{t-1}(4\alpha_{\mathcal{Z}}+1)^m\hat{E}_{z}[L(Z_{t-m-1})]\right)\nonumber \\
&+  \alpha_{\mathcal{Z}} \sum_{t=0}^{k-1}\beta^{t+1}\|v_{k-t-1}\|_{BL}\left(\hat{E}_{z}\left[L(Z_t)\right]+ 3 \alpha_{\mathcal{Z}} \sum_{m=0}^{t-1}(4\alpha_{\mathcal{Z}}+1)^m\hat{E}_{z}[L(Z_{t-m-1})]\right)\Bigg).
\end{align*}

Using the fact that $\|L\|_\infty\leq2$ we can write
\begin{align*}
\|g\|_\infty &\leq  \lim_{k\to\infty}\Bigg(\alpha_c \sum_{t=0}^{k-1}\beta^{t}\left( 2+ 3\alpha_{\mathcal{Z}} \sum_{m=0}^{t-1}2(4\alpha_{\mathcal{Z}}+1)^m\right) \\
& \qquad +  \alpha_{\mathcal{Z}} \sum_{t=0}^{k-1}\beta^{t+1}\|v_{k-t-1}\|_{BL}\left(2+ 3\alpha_{\mathcal{Z}} \sum_{m=0}^{t-1}2(4\alpha_{\mathcal{Z}}+1)^m\right)\Bigg)\\
&\leq
2 \left(\alpha_c + \beta\alpha_{\mathcal{Z}}\|J_\beta^*\|_{BL}\right)\frac{1}{1-\beta(4\alpha_{\mathcal{Z}}+1)}.
\end{align*}

Next, we show that when defined as a function of $z$, $\|\hat{E}_z[L(Z_t)]\|_{BL}\leq 3\frac{\alpha_{\mathcal{Z}}^{t+1}-1}{\alpha_{\mathcal{Z}}-1}$ 
We follow an inductive approach. For $t=1$:
\begin{align*}
\big|\hat{E}_z\left[L(Z_1)\right]-\hat{E}_{\hat{z}}\left[L(Z_1)\right]\big|&\leq\sup_u\left|\int L(z_1)\eta(dz_1|z,u)-\int L(z_1)\eta(dz_1|\hat{z},u)\right|\\
&\leq\|L\|_{BL}\alpha_{\mathcal{Z}}\rho_{BL}(z,\hat{z})\leq 3\alpha_{\mathcal{Z}}\rho_{BL}(z,\hat{z})
\end{align*}
where $\|L\|_{BL}\leq 3$ follows from (\ref{L_bl_bound}). Hence we have that $\|\hat{E}_z[L(Z_1)]\|_{BL}\leq 3+3\alpha_{\mathcal{Z}}$.  Now we assume the claim holds for $t-1$ and focus on the case for $t$:
\begin{align*}
&\big|\hat{E}_z\left[L(Z_t)\right]-\hat{E}_{\hat{z}}\left[L(Z_t)\right]\big|\\
&=\sup_{u_0}\int\dots\sup_{u_{t-2}}\int \sup_{u_{t-1}}\int L(z_t)\eta(dz_t|z_{t-1},u_{t-1})\eta(dz_{t-1}|z_{t-2},u_{t-2})\dots \eta(dz_1|z,u_0)\\
&\quad - \sup_{u_0}\int\dots\sup_{u_{t-2}}\int \sup_{u_{t-1}}\int L(z_t)\eta(dz_t|z_{t-1},u_{t-1})\eta(dz_{t-1}|z_{t-2},u_{t-2})\dots \eta(dz_1|\hat{z},u_0)\\
&\leq \sup_{u_0}\left(\int \hat{E}_{z_1}\left[L(Z_t)\right]\eta(dz_1|z,u_0)- \int \hat{E}_{z_1}\left[L(Z_t)\right]\eta(dz_1|\hat{z},u_0)\right)\\
&\leq \alpha_{\mathcal{Z}} \rho_{BL}(z,\hat{z})\|\hat{E}_z[L(Z_{t-1})]\|_{BL}\leq  \alpha_{\mathcal{Z}} \rho_{BL}(z,\hat{z})3\frac{\alpha_{\mathcal{Z}}^{t}-1}{\alpha_{\mathcal{Z}}-1}.
\end{align*}
We then have that 
\begin{align*}
\|\hat{E}_z[L(Z_{t})]\|_{BL}\leq \|L\|_{\infty}+ 3\alpha_{\mathcal{Z}}\frac{\alpha_{\mathcal{Z}}^{t}-1}{\alpha_{\mathcal{Z}}-1}\leq 3+3\alpha_{\mathcal{Z}}\frac{\alpha_{\mathcal{Z}}^{t}-1}{\alpha_{\mathcal{Z}}-1}\leq  3\frac{\alpha_{\mathcal{Z}}^{t+1}-1}{\alpha_{\mathcal{Z}}-1}.
\end{align*}
Thus, we can write that
\begin{align*}
\big|\hat{E}_z\left[L(Z_t)\right]-\hat{E}_{\hat{z}}\left[L(Z_t)\right]\big|\leq3\alpha_{\mathcal{Z}}\frac{\alpha_{\mathcal{Z}}^{t}-1}{\alpha_{\mathcal{Z}}-1}\rho_{BL}(z,\hat{z}).
\end{align*}

Hence for any $z,\hat{z}$
\begin{align*}
&|g(z)-g(\hat{z})|  \leq \lim_{k\to\infty}  \bigg(\alpha_c \sum_{t=0}^{k-1}\beta^{t}\big( \big|\hat{E}_z\left[L(Z_t)\right]-\hat{E}_{\hat{z}}\left[L(Z_t)\right]\big|  \\
&\hspace{25mm}+ 3\alpha_{\mathcal{Z}} \sum_{m=0}^{t-1}(4\alpha_{\mathcal{Z}}+1)^m \left|\hat{E}_z\left[L(Z_{t-m-1})\right]-\hat{E}_{\hat{z}}\left[L(Z_{t-m-1})\right]\right|\bigg) \\
& \hspace{25mm}+  \alpha_{\mathcal{Z}} \sum_{t=0}^{k-1}\beta^{t+1}\|v_{k-t-1}\|_{BL}\bigg(  \left|\hat{E}_z\left[L(Z_t)\right]-\hat{E}_{\hat{z}}\left[L(Z_t)\right]\right| \\
&\hspace{25mm} + 3\alpha_{\mathcal{Z}} \sum_{m=0}^{t-1}(4\alpha_{\mathcal{Z}}+1)^m \left|\hat{E}_z\left[L(Z_{t-m-1})\right]-\hat{E}_{\hat{z}}\left[L(Z_{t-m-1})\right]\right|\bigg) \bigg)\\
&\leq \lim_{k\to\infty} \Bigg( \alpha_c \sum_{t=0}^{k-1}\beta^{t}\left(  3\alpha_{\mathcal{Z}}\frac{\alpha_{\mathcal{Z}}^{t}-1}{\alpha_{\mathcal{Z}}-1}\rho_{BL}(z,\hat{z})+ 3\alpha_{\mathcal{Z}} \sum_{m=0}^{t-1}(4\alpha_{\mathcal{Z}}+1)^m  3\alpha_{\mathcal{Z}}\frac{\alpha_{\mathcal{Z}}^{t-m-1}-1}{\alpha_{\mathcal{Z}}-1}\rho_{BL}(z,\hat{z})\right) \\
&+  \alpha_{\mathcal{Z}} \sum_{t=0}^{k-1}\beta^{t+1}\|v_{k-t-1}\|_{BL}\left(3\alpha_{\mathcal{Z}}\frac{\alpha_{\mathcal{Z}}^{t}-1}{\alpha_{\mathcal{Z}}-1}\rho_{BL}(z,\hat{z})+ 3\alpha_{\mathcal{Z}} \sum_{m=0}^{t-1}(4\alpha_{\mathcal{Z}}+1)^m 3\alpha_{\mathcal{Z}}\frac{\alpha_{\mathcal{Z}}^{t-m-1}-1}{\alpha_{\mathcal{Z}}-1}\rho_{BL}(z,\hat{z})\right)\Bigg)\\
&\leq \frac{3\alpha_{\mathcal{Z}}}{\alpha_{\mathcal{Z}}-1}(\alpha_c+\beta\alpha_{\mathcal{Z}}\|J_\beta\|_{BL})\left(\frac{1}{1-\beta\alpha_{\mathcal{Z}}}+\frac{3\alpha_{\mathcal{Z}}}{(1-\beta(4\alpha_{\mathcal{Z}}+1))^2}\right)\rho_{BL}(z,\hat{z}).
\end{align*}
Hence, we have that 
\begin{align*}
&\|g\|_{BL}\leq  (\alpha_c+\beta\alpha_{\mathcal{Z}}\|J_\beta\|_{BL})\left(\frac{2}{1-\beta(4\alpha_{\mathcal{Z}}+1)}+\frac{3\alpha_{\mathcal{Z}}}{1-\beta\alpha_{\mathcal{Z}}}+\frac{9\alpha_{\mathcal{Z}}^2}{(1-\beta(4\alpha_{\mathcal{Z}}+1))^2}\right).
\end{align*}
\end{proof}

\begin{lemma}\label{lip_val}
\begin{itemize}
\item[i.]
Under Assumption \ref{belief_reg}, if $\mathcal{Z}$ is metrized with $\rho_{BL}$, we have that
\begin{align*}
\|J^*_\beta\|_{BL}\leq \frac{1}{1-\beta\alpha_{\mathcal{Z}}}\left(\frac{\|\tilde{c}\|_\infty}{1-\beta}+\alpha_{\tilde{c}}\right).
\end{align*}
\item [ii.] Without any assumption, if $\mathcal{Z}$ is metrized with total variation distance, we have that
\begin{align*}
\|J^*_\beta\|_{BL}\leq \frac{2-\beta}{(1-\beta)(1-\beta\alpha_{\mathcal{Z}})}\|c\|_\infty
\end{align*}
\end{itemize}
\end{lemma}

\begin{proof}
We start with the first part, for any $x,y \in Z$
\begin{align*}
|J_\beta^*(x)-J_\beta^*(y)|&\leq \sup_u\bigg(|\tilde{c}(x,u)-\tilde{c}(y,u)|+\beta |\int J_\beta^*(z)\eta(dz|x,u)-\int J_\beta^*(z)\eta(dz|y,u)|\bigg)\\
&\leq \alpha_{\tilde{c}}\rho_{BL}(x,y)+\beta\|J_\beta^*\|_{BL}\alpha_{\mathcal{Z}}\rho_{BL}(x,y).
\end{align*}
Thus, we can write that
\begin{align*}
\|J_\beta^*\|_{BL}\leq \|J_\beta^*\|_\infty+\alpha_{\tilde{c}}+\beta\alpha_{\mathcal{Z}}\|J_\beta^*\|_{BL}\leq \frac{\|\tilde{c}\|_\infty}{1-\beta}+\alpha_{\tilde{c}}+\beta\alpha_{\mathcal{Z}}\|J_\beta^*\|_{BL}.
\end{align*}
Hence, rearranging the terms, we can write (provided $\beta\alpha_{\mathcal{Z}}< 1$)
\begin{align*}
\|J_\beta^*\|_{BL}\leq\frac{1}{1-\beta\alpha_{\mathcal{Z}}}\left(\frac{\|\tilde{c}\|_\infty}{1-\beta}+\alpha_{\tilde{c}}\right).
\end{align*}
For the second part, similar arguments lead to the following bound
\begin{align*}
\|J_\beta^*\|_{BL}\leq \|J^*_\beta\|_\infty+ \|c\|_\infty +\beta\alpha_{\mathcal{Z}}\|J_\beta^*\|_{BL},
\end{align*}
hence the result follows after rearranging the terms and noting that $\|J_\beta^*\|_\infty\leq \frac{\|c\|_\infty}{1-\beta}$.
\end{proof}

\bibliographystyle{plain}
\bibliography{AliBibliography,SerdarBibliography_acc,references,references_acc,SerdarBibliography}

\end{document}